\documentclass[12 pt]{amsart}

\usepackage{amsfonts, amsmath, amssymb, amsthm, amsopn, latexsym, hyperref, graphicx, mathrsfs, verbatim, enumerate, stmaryrd} % Standard packages
\usepackage{subfig} %For easily making subfigures of figures.
\usepackage[all]{hypcap} %So that hyperlinks to figures go to the figure and not the caption
\usepackage{pst-all, multido} %Package to draw figures
\usepackage[margin=1in]{geometry} %Correct margins for a5 paper
\usepackage{Style} % My personal commands
\begin{document}

\setcounter{tocdepth}{1}

\renewcommand{\qedsymbol}{$\square$}

%%%%%%%%%%%Theorem Environments%%%%%%%%%%%

\newtheorem*{equid}{Theorem}
\newtheorem*{conj}{Conjecture}
\newtheorem{letthm}{Theorem}
\renewcommand*{\theletthm}{\Alph{letthm}}
\newtheorem{thm}{Theorem}[section]
\newtheorem{lem}[thm]{Lemma}
\newtheorem{cor}[thm]{Corollary}
\newtheorem{prop}[thm]{Proposition}
\theoremstyle{definition}
\newtheorem{rem}[thm]{Remark}
\newtheorem*{War}{Warning}
\newtheorem{Def}[thm]{Definition}
\newtheorem{Not}[thm]{Notation}
\newtheorem{Ass}[thm]{Assumption}
\newtheorem{ex}[thm]{Example}
\newtheorem*{exs}{Examples}
\newtheorem{obs}[thm]{Observation}
\newtheorem*{Ack}{Acknowledgements}

%%%%%%%%%%%%%%%%%%%%%%%%%%%%%%%%%

%\frontmatter

\title[Equidistribution of preimages for maps of good reduction]{Equidistribution of preimages over nonarchimedean fields for maps of good reduction}
\author[W. Gignac]{William Gignac}
\address{Department of Mathematics, University of Michigan, 530 Church St., Ann Arbor, MI 48109, USA}
\email{wgignac@umich.edu}
\date{\today}
\maketitle

\begin{abstract} In this article we prove an analogue of the equidistribution of preimages theorem from complex dynamics for maps of good reduction over nonarchimedean fields. While in general our result is only a partial analogue of the complex equidistribution theorem, for most maps of good reduction it is a complete analogue. In the particular case when the nonarchimedean field in question is equipped with the trivial absolute value, we are able to supply a strengthening of the theorem, namely that the preimages of any \emph{tame} valuation equidistribute to a canonical measure.
\end{abstract}

%\tableofcontents

\section{Introduction} 

Ergodic methods play a central role in studying the dynamics of morphisms $f\colon X\to X$ of complex projective varieties. At the heart of these methods are equidistribution results, which allow one to construct dynamically interesting $f$-invariant probability measures on $X$. The most important of these  is the \emph{equidistribution of preimages} theorem, which will be the focus of the present article.

\begin{equid}[Equidistribution of preimages over $\C$] Let $X$ be an irreducible complex projective variety, and let $f\colon X\to X$ be a polarized dynamical system of degree $d\geq 2$. Then there is an $f$-invariant probability measure $\mu_f$ on $X$ and a proper Zariski closed subset $\ms{E}_f\subset X$ such that the iterated preimages of any $x\in X\smallsetminus \ms{E}_f$ equidistribute to $\mu_f$.
\end{equid}

Polarized dynamical systems are, roughly speaking, those which arise from endomorphisms of projective space. More precisely, $f$ is polarized if it is obtained by restricting a morphism $F\colon \pr^r_\C\to \pr^r_\C$ of degree $d$ to an invariant subvariety $X\subseteq \pr^r_\C$.  See \S4 for further discussion.

Brolin was the first to observe this phenomenon; he proved the theorem in the case where $X = \pr^1_\C$ and $f$ is a polynomial, using potential theoretic methods \cite{MR0194595}. Two decades later the result was extended independently by Ljubich \cite{MR741393} and Freire-Lopes-Ma\~{n}\'{e} \cite{MR736568} to rational maps on $\pr^1_\C$. Following earlier work by Fornaess-Sibony \cite{MR1369137}, the theorem was proved for endomorphisms of $\pr^r_\C$ by Briend-Duval \cite{MR1863737} and Dinh-Sibony \cite{MR1992375}. In the above generality, equidistribution was proved more recently by Dinh-Sibony \cite{MR2468484}. Along the way, similar results regarding the equidistribution of pullbacks of positive closed $(1,1)$-currents have been shown, see for instance \cite{MR1760844}, \cite{MR2017143}, \cite{MR2032940}, \cite{MR2468484}, and \cite{Rodrigo}.

The equidistribution theorem does not carry over in any obvious way to endomorphisms  of varieties over nonarchimedean fields. Because of their topological flaws, many of the analytic techniques used in the complex setting are not available over these fields. In particular, the notion of Radon measure does not make sense over nonarchimedean fields, so one cannot speak of weak convergence of measures, making equidistribution problems ill-posed.

To overcome these topological obstacles, one must eschew working on projective varieties  $X$ over a nonarchimedean field $K$, and instead work over their Berkovich analytification $X^\an$. The spaces $X^\an$ are compact Hausdorff and naturally contain $X$ as a subspace. Moreover, any endomorphism $f$ of $X$ extends to an endomorphism $f^\an$ of $X^\an$, allowing one to study the dynamics of $f$ by working in $X^\an$. Because $X^\an$ is compact Hausdorff, one has the notion of Radon measure, making it possible to study equidistribution problems. 

A nonarchimedean analogue of the equidistribution theorem has recently been proved for rational maps $f\colon \pr^{1,\an}_K\to \pr^{1,\an}_K$ by Favre and Rivera-Letelier \cite{MR2578470}, see also  \cite[\S5]{Jonsson}.  A quantitative strengthening of this theorem has recently been obtained by Okuyama \cite{Okuyama}, who has also studied nonarchimedean equidistribution of repelling points \cite{MR2885787}. Their methods are potential theoretic in nature, and, at least for the moment, do not extend to endomorphisms of $\pr^{r,\an}_K$ for $r>1$. In a separate result, Chambert-Loir has constructed an $f$-invariant probability measure $\mu_f$ on $X^\an$ for polarized dynamical systems $f\colon X^\an\to X^\an$ of a projective variety $X$ over $K$ \cite{MR2244803}, and Yuan has used this construction to prove an equidistribution result for points of small height for morphisms defined over number fields \cite{MR2425137}, see also the related results in  \cite{MR1427622}, \cite{MR2244226}, \cite{MR2221116}, \cite{MR2244803}, \cite{MR2457191}, and \cite{MR2506588}. It is not clear from these works, however, that one has equidistribution of preimages to $\mu_f$ for such a map. Nonetheless, we conjecture the following.

\begin{conj}\label{conjecture} Let $K$ be an algebraically closed complete nonarchimedean field, allowing the possibility of a trivial absolute value, and let $X$ be an irreducible projective variety over $K$. Suppose $f\colon X\to X$ is a flat polarized dynamical system of degree $d\geq 2$. Let $x\in X^\an$ be any point, and let $Y\subseteq X$ be the smallest totally invariant Zariski closed set such that $x\in Y^\an\subseteq X^\an$. Let $Y_0\subseteq Y$ be the unique component of $Y$ such that $x\in Y_0^\an$, and let $m\geq 1$ be an integer such that $f^{-m}(Y_0) = Y_0$. Then the iterated $f^m$-preimages of $x$ equidistribute the Chambert-Loir measure associated to the dynamical system $f^m\colon Y_0^\an\to Y_0^\an$.
\end{conj}

Though we will use it several times in this article, it is possible that the flatness assumption in the conjecture is unnecessary. It should be noted that any polarized endomorphism of a smooth variety is automatically flat. In general, the interaction between flatness and dynamics is not entirely understood.  This is the subject of the recent work  \cite{MZMS}, where it is proved that $f$ is flat at a superattracting periodic point $x\in X$ if and only if $X$ is smooth that $x$.

The goal of this article is to study the validity of this conjecture in the case where $f$ is a map of good reduction. Our main result is the following, which we state, for simplicity, only in the case of endomorphisms of projective space.

\begin{letthm}\label{thmA} Let $K$ be an algebraically closed complete nonarchimedean field, possibly with trivial absolute value, and let $k$ be the residue field of $K$. Let $f\colon \pr^r_K\to \pr^r_K$ be a  morphism of degree $d\geq 2$ with good reduction, and let $\tilde{f}\colon \pr^r_k\to \pr^r_k$ be the reduction of $f$. Suppose that the characteristic of $k$ does not divide $d$.
Then \begin{enumerate}
\item[$1.$] there is a maximal proper Zariski closed subset $\ms{E}\subset\pr^r_k$ such that $\tilde{f}^{-1}(\ms{E}) = \ms{E}$.
\item[$2.$] for every $x\in \pr^{r,\an}_K$ whose reduction does not lie in $\ms{E}$, the iterated preimages of $x$ equidistribute to the Dirac probability measure supported at the Gauss point of $\pr^{r,\an}_K$.
\end{enumerate} In particular, if $\ms{E} = \varnothing$, the \hyperref[conjecture]{conjecture} holds for $f$.
\end{letthm}

The case when $K$ is equipped with the trivial absolute value deserves special note here, for in this case all morphisms $f\colon \pr^r_K\to \pr^r_K$ have good reduction. Moreover, we will prove (see \hyperref[genericity]{Theorem~\ref*{genericity}}\color{black}) that generic morphisms $f$ satisfy the condition $\ms{E} = \varnothing$. Thus \hyperref[thmA]{Theorem~\ref*{thmA}} gives the full equidistribution theorem for most endomorphisms $f$ of $\pr^r_K$ when $K$ is trivially valued. However, in the case when $\ms{E} \neq\varnothing$, \hyperref[thmA]{Theorem~\ref*{thmA}} is strictly weaker than the conjecture, as there will be many points $x\in \pr^{r,\an}_K$ whose reduction lies in $\ms{E}$, but whose preimages still equidistribute to the Dirac probability measure at the Gauss point. Our next main theorem illustrates this.

\begin{letthm}\label{thmB} Suppose $K$ is a trivially valued algebraically closed field and $f\colon \pr^r_K\to \pr^r_K$ is a morphism of degree $d\geq 2$. Assume $\mathrm{char}(K)\nmid d$. Then the preimages of any divisorial point $x\in \pr^{r,\an}_K$ equidistribute to the Dirac mass at the Gauss point of $\pr^{r,\an}_K$.
\end{letthm}

A \emph{divisorial} point $x\in \pr^{r,\an}_K$ is a point corresponding to a valuation on the function field of $\pr^r_K$ which is proportional to the order of vanishing along an exceptional prime divisor of some blowup $X\to \pr^r_K$ of $\pr^r_K$. Such points are dense in $\pr^{r,\an}_K$. We will actually prove the theorem holds of a more general class of points $x\in \pr^{r,\an}_K$, which we call \emph{tame} points. See \S9 for details.

We expect \hyperref[thmB]{Theorem~\ref*{thmB}} to hold in the case when $K$ is nontrivially valued, as well, and it is even possible that our proof of the theorem can be carried out in this case. However, a direct translation of the proof would require intersection theory over the valuation ring $K^\circ$, a complication we mean to avoid here.

Though we do not prove the  conjecture for all morphisms $f$ of good reduction, we are able to give a simple argument for a slightly weaker equidistribution result, which at the very least supplies some evidence for the veracity of the conjecture when $K$ is trivially valued.

\begin{letthm}\label{thmC} Let $K$ be a trivially valued algebraically closed field, and let $f\colon \pr^r_K\to \pr^r_K$ be a morphism of degree $d\geq 2$. Assume that $\mathrm{char}(K)\nmid d$, and, moreover, that every totally invariant cycle for $f$ is superattracting. Let $x\in \pr^{r,\an}_K$ be such that $x\notin Y^\an$ for any proper totally invariant Zariski closed set $Y\subset \pr^r_K$. Then the Cesaro means \[
\frac{1}{n}\sum_{i=0}^{n-1} d^{-ir}f^{i*}\delta_x\] converge weakly to the Dirac mass at the Gauss point of $\pr^{r,\an}_K$ as $n\to \infty$.
\end{letthm}

The term \emph{superattracting} here means the following. If $V\subseteq \pr^r_K$ is an irreducible subvariety which is periodic in the sense that $f^s(V) = V$ for some $s\geq 1$, then this periodic cycle is superattracting when there is an $n\geq 1$ such that $f^{ns*}\mf{m}_V\subseteq \mf{m}_V^2$, where $\mf{m}_V$ is the maximal ideal of the local ring $\mc{O}_{\pr^r_K, V}$. This generalizes the standard notion of a superattracting cycle in dimension $r = 1$.

It should be noted that the proofs we give for \hyperref[thmA]{Theorems~\ref*{thmA}}, \ref{thmB}, and \ref{thmC} use heavily that the map $f$ under consideration has good reduction. Proving the \hyperref[conjecture]{conjecture} for general $f$ would require new tools.

The majority of this article will be spent proving \hyperref[thmA]{Theorem~\ref*{thmA}}. The idea behind the proof is simple: we will \emph{approximate} the dynamics of $f$ on $\pr^{r,\an}_K$ by the dynamics of the reduced map $\tilde{f}$ on $\pr^r_k$. Specifically, we will prove a version of \hyperref[thmA]{Theorem~\ref*{thmA}} for $\tilde{f}\colon \pr^r_k\to \pr^r_k$, and then lift it to the Berkovich setting via the reduction map. As a consequence, roughly the first half of this article will be spent not in the Berkovich setting, but in the classical algebro-geometric setting of varieties over $k$.

In \S2 through \S4 we develop the tools needed to prove the equidistribution theorem in the classical algebro-geometric setting. In \S2, we will discuss two multiplicities associated to finite endomorphisms of projective varieties. In \S3, we briefly develop a language of Borel measures on projective varieties, and, crucially, the notion of a pull-back of such a measure via a finite morphism. In \S4 we adapt common techniques for detecting totally invariant behavior from the setting of complex dynamics to dynamics over the (arbitrary) algebraically closed field $k$. It is here that we prove statement (1) in \hyperref[thmA]{Theorem~\ref*{thmA}}, the existence of a finite exceptional set, and here that the assumption $\mathrm{char}(k)\nmid d$ comes into play. We will also show in \S4 that generic morphisms $f$ have empty exceptional set.

Finally, in \S5, we prove the equidistribution theorem in the classical algebro-geometric setting, namely, for noninvertible polarized endomorphisms of projective varieties over $k$. This theorem is the technical heart of \hyperref[thmA]{Theorem~\ref*{thmA}}, but is also interesting in its own right as a nearly complete analogue of the complex equidistribution theorem in the purely algebraic setting.

Beginning in \S6, we move on to the Berkovich setting. In \S6 we briefly review the Berkovich analytification of varieties over nonarchimedean fields $K$, discuss multiplicities associated to finite morphisms of analytic varieties, and define the pull-back of Radon measures on these varieties. In \S7 we   will discuss models of analytic spaces, the notion of reduction, and define precisely what we will mean by a \emph{map of good reduction}. Finally, \S8 is devoted to the proof of \hyperref[thmA]{Theorem~\ref*{thmA}} and \S9 to the proofs of \hyperref[thmB]{Theorems~\ref*{thmB}} and \ref{thmC}.

\subsection*{Acknowledgements} I want to wholeheartedly thank my advisor Mattias Jonsson for his insight, encouragement, and unflagging support throughout the course of this project. I would also like to thank Charles Favre and Matt Baker for useful conversations on this and related topics, and the referee for useful commentary. This work was supported by the grants DMS-0602191, DMS-0901073 and DMS-1001740. Also, a portion of this work was done while the author was a visiting researcher at the Institute for Computational and Experimental Research Mathematics.

\section{Multiplicities associated to finite morphisms}

In this section we briefly review some basic algebro-geometric concepts we will need later. The setup for the entirety of this section is as follows. Let $k$ be an algebraically closed field, and let $f\colon X\to Y$ be a finite surjective flat morphism between two irreducible varieties over $k$. Many of the statements in this section hold in greater generality, but for the sake of concreteness we will stick to this very specific situation. The flatness assumption will be needed throughout this article. We note, however, that in the case where $X$ and $Y$ are both smooth, any finite surjective morphism $f\colon X\to Y$ is flat. We will always regard $X$ and $Y$ as schemes, thus allowing for non-closed points.

The goal of this section is to define two multiplicities associated to $f$, and to discuss their relationship. The first assigns to every point $x\in X$ an integer $m_f(x)$ that we will call the \emph{multiplicity} of $f$ at $x$. The multiplicity function $m_f\colon X\to \N$ will be used in \S3 to define the pull-back of measures on varieties. The second assigns to each point $x\in X$ an integer $v_f(x)$, which will be called the \emph{generic multiplicity} of $f$ at $x$. The generic multiplicity function $v_f\colon X\to \N$  will be used in \S4 to detect totally invariant behavior in dynamical systems.

We begin by fixing notation. The structure sheaves of $X$ and $Y$ will be denoted $\mc{O}_X$ and $\mc{O}_Y$. If $x\in X$ is a point, then $\mf{m}_x$ denotes the maximal ideal of the local ring $\mc{O}_{X,x}$, and $\kappa(x)$ denotes its residue field; similar notation is used for points of $Y$. Recall that the \emph{degree} of $f$ is the degree of the field extension $k(X)/f^*k(Y)$, where $k(X)$ and $k(Y)$ are the function fields of $X$ and $Y$. This degree will be written $[X:_fY]$. Similarly, $[X:_fY]_s$ and $[X:_fY]_i$ will denote the separable and purely inseparable factors of this degree. 

The main definition of the section is the following.

\begin{Def} Let $x\in X$ and $y = f(x)$. The \emph{multiplicity} of $f$ at $x$ is the integer \[
m_f(x) :=  \dim_{\kappa(y)}(\mc{O}_{X,x}/\mf{m}_y\mc{O}_{X,x}),\] where $\mc{O}_{X,x}$ is viewed as an $\mc{O}_{Y,y}$-module via $f$. Let $E = \ol{\{x\}}$ and $F = \ol{\{y\}}$. Then the \emph{generic multiplicity} of $f$ at $x$ is the integer \[
v_f(x) := [E:_f F]_i\times \mathrm{length}_{\mc{O}_{X,x}}(\mc{O}_{X,x}/\mf{m}_y\mc{O}_{X,x}).\] It will sometimes be convenient to write $m_f(E)$ and $v_f(E)$ in place of $m_f(x)$ and $v_f(x)$.
\end{Def}

\begin{lem}\label{lem:lengths} Let $(A,\mf{m})$  and $(B,\mf{n})$ be Noetherian local rings, with $B$ a finite flat $A$-module. Let $\mf{a}$ be an $\mf{m}$-primary ideal of $A$. Then the following identities hold: \begin{enumerate}
\item[$1.$] $\mathrm{length}_{B}(B/\mf{a}B) = \mathrm{length}_A(A/\mf{a})\mathrm{length}_B(B/\mf{m}B)$.
\item[$2.$] $\mathrm{length}_A(B/\mf{a}B) = \mathrm{length}_B(B/\mf{a}B)[B/\mf{n} : A/\mf{m}]$.
\end{enumerate}
\end{lem}
\begin{proof} (1) Let $A/\mf{a} = I_0\supsetneq I_1\supsetneq \cdots\supsetneq I_N = 0$ be a composition series of $A/\mf{a}$. Since $A$ is local, the successive quotients $I_i/I_{i+1}$ are each isomorphic to $A/\mf{m}$.  Because $B$ is a flat over $A$, one obtains a filtration $
B/\mf{a}B = B\otimes_A I_0\supseteq B\otimes_A I_1\supseteq\cdots\supseteq B\otimes_A I_N = 0$ of $B/\mf{a}B$, whose successive quotients are $(B\otimes_A I_i)/(B\otimes_A I_{i+1})\cong B\otimes_A (I_i/I_{i+1})\cong B\otimes_A A/\mf{m}\cong B/\mf{m}B$. Thus $\mathrm{length}_B(B/\mf{a}B) = N\times\mathrm{length}_B(B/\mf{m}B)$, as desired.

(2) Now fix a composition series $B/\mf{a}B = J_0\supsetneq J_1\supsetneq \cdots\supsetneq J_M = 0$ of $B/\mf{a}B$ as a $B$-module. Since $B$ is local, the quotients $J_i/J_{i+1}$ are all isomorphic to $B/\mf{n}$. Thus \[\mathrm{length}_A(B/\mf{a}B) = M\times \mathrm{length}_A(B/\mf{n}) = M\times [B/\mf{n} : A/\mf{m}],\] as desired.
\end{proof}

\begin{prop}\label{prop:multiplicity_relation} The multiplicity functions $m_f\colon X\to \N$ and $v_f\colon X\to \N$ are related as follows. Let $x\in X$ and $y = f(x)$. Let $E = \ol{\{x\}}$ and $F = \ol{\{y\}}$. Then \[m_f(x) = v_f(x)[E:_fF]_s.\] In particular, if $x$ is a closed point, then $m_f(x) = v_f(x)$.
\end{prop}
\begin{proof} Applying \hyperref[lem:lengths]{Lemma~\ref*{lem:lengths}(2)} to the case where $A = \kappa(y)$, $\mf{a} = 0$, and $B = \mc{O}_{X,x}/\mf{m}_y\mc{O}_{X,x}$ yields $\dim_{\kappa(y)}(\mc{O}_{X,x}/\mf{m}_y\mc{O}_{X,x}) = \mathrm{length}_{\mc{O}_{X,x}}(\mc{O}_{X,x}/\mf{m}_y\mc{O}_{X,x})\times [\kappa(x):\kappa(y)]$, which is exactly the desired identity $m_f(x) = v_f(x)[E:_f F]_s$. If $x$ and $y$ are closed points, then $\kappa(x) = \kappa(y) = k$, since $k$ is algebraically closed. Thus $[E:_f F] = 1$ in this case, so that $m_f(x) = v_f(x)$.
\end{proof}

\begin{thm}\label{thm:preimage_count} Every point $y\in Y$ has exactly $[X:_fY]$ preimages when counted according to their multiplicity. That is, $[X:_fY] = \sum_{f(x) = y} m_f(x)$.
\end{thm}
\begin{proof} Since $f$ is finite and flat, $f_*\mc{O}_X$ is a locally free $\mc{O}_Y$-module of some rank  $r<\infty$. The fiber of $f_*\mc{O}_X$ at a point $y\in Y$ is \[
(f_*\mc{O}_X)_y/\mf{m}_y(f_*\mc{O}_X)_y \cong \bigoplus_{f(x) = y} \mc{O}_{X,x}/\mf{m}_y\mc{O}_{X,x}.\] Comparing the $\kappa(y)$-dimension of both sides of this isomorphism, we see $r = \sum_{f(x) = y}m_f(x)$. In the special case where $y$ is the generic point of $Y$, this identity yields $r = [X:_fY]$.
\end{proof}

\begin{prop}\label{prop:multiplicativity} Suppose that $g\colon Y\to Z$ is another finite surjective flat morphism between irreducible varieties. Let $x\in X$ and $y = f(x)$. Then the multiplicity and generic multiplicity are multiplicative in the sense that $m_{g\circ f}(x) = m_f(x)m_g(y)$ and $v_{g\circ f}(x) = v_f(x)v_g(y)$.
\end{prop}
\begin{proof} By \hyperref[prop:multiplicity_relation]{Proposition~\ref*{prop:multiplicity_relation}}, it is enough to show that the generic multiplicity is multiplicative. Moreover, since degrees of inseparability for field extensions are multiplicative, it suffices to show that \[
\mathrm{length}_{\mc{O}_{X,x}}(\mc{O}_{X,x}/\mf{m}_{g(y)}\mc{O}_{X,x}) =\mathrm{length}_{\mc{O}_{X,x}}(\mc{O}_{X,x}/\mf{m}_y\mc{O}_{X,x})\times \mathrm{length}_{\mc{O}_{Y,y}}(\mc{O}_{Y,y}/\mf{m}_{g(y)}\mc{O}_{Y,y}).\] This is exactly \hyperref[lem:lengths]{Lemma~\ref*{lem:lengths}(1)}.
\end{proof}

\begin{thm}[Lejeune-Jalabert and Teissier]\label{LJT} There is a coherent sheaf $\mc{F}$ on $X$ whose fiber dimensions are given by $v_f$. As a consequence, the generic multiplicity function $v_f\colon X\to~\N$ is Zariski upper semicontinuous.
\end{thm}
\begin{proof}[Sketch] The sheaf $\mc{F}$ is constructed as follows. Let $Z = X\times_f X$, and let $\mc{I}$ denote the ideal sheaf of the diagonal $\Delta\subseteq Z$. Let $\pi\colon Z\to X$ be the projection onto the first coordinate. We then set $\mc{F} = \pi_*(\mc{O}_Z/\mc{I}^n)$, where $n$ is a large enough integer ($n\geq [X:_fY]$ will suffice). The fiber dimension of $\mc{F}$ at any (not necessarily closed) point $x$ is computed in proposition 4.7 of \cite{MR0379897} to be exactly $v_f(x)$. The upper semicontinuity statement is then a consequence of Nakayama's Lemma, see \cite[Example III.12.7.2]{MR0463157}.
\end{proof}

\hyperref[LJT]{Theorem~\ref*{LJT}} gives the reason behind the name \emph{generic multiplicity} of $v_f$. Indeed, if $x\in X$ is any point, then the upper semicontinuity of $v_f$ implies that $v_f(x) = v_f(z) = m_f(z)$ for a nonempty Zariksi open subset of closed points $z$ specializing $x$. That is, $v_f(x)$ is given by the multiplicity of $f$ at general closed points specializing $x$.

\begin{lem}\label{lem:topdeg} There is a nonempty Zariski open set $U\subseteq Y$ such that any closed point $y\in U$ has exactly $[X:_fY]_s$ preimages in $X$.
\end{lem}
\begin{proof} Without loss of generality, we may assume that $X$ and $Y$ are both affine and smooth, with coordinate rings $k[Y]\subseteq k[X]$. Let $L$ be the unique intermediate field $k(Y)\subseteq L\subseteq k(X)$ such that $L$ is separable over $k(Y)$ and $k(X)$ is purely inseparable over $L$. Let $A$ be the integral closure of $k[Y]$ in $L$. Then $A$ is the coordinate ring of some irreducible affine variety $Z$ (see \cite[Corollary 13.13]{MR1322960}), and the inclusions $k[Y]\subseteq A\subseteq k[X]$ induce morphisms

\bigskip

 \[\begin{psmatrix}
X & Z & Y
\psset{arrows = ->, nodesep = 3pt}
\ncline{1,1}{1,2}_g
\ncline{1,2}{1,3}_h
\ncarc[arcangle=30]{1,1}{1,3}^f
\end{psmatrix}\]

\bigskip

\noindent Since $k(X)$ is purely inseparable over $L$, each maximal ideal of $A$ has only one maximal ideal in $k[X]$ lying over it, so $g$ is injective. It therefore suffices to prove the theorem for $h$, i.e., we may assume without loss of generality that $[X:_fY]_i = 1$. In this case, the lemma is proved in \cite[\S II.6.3]{MR1328833}.
\end{proof}

\begin{prop}\label{prop:nice_neighborhood} Let $E\subseteq X$ be an irreducible closed subvariety, and set $F = f(E)$. Let $z$ denote the generic point of $E$. Then there is a nonempty Zariski open subset $U$ of $F$ such that for all $y\in U$, \[m_f(z) = \sum_{x\in f^{-1}(y)\cap E}m_f(x).\]
\end{prop}
\begin{proof} Using \hyperref[lem:topdeg]{Lemma~\ref*{lem:topdeg}} and the fact that $v_f$ is upper semicontinuous, there is a nonempty open subset $U$ of $F$ with the following two properties: \begin{enumerate}
\item[1.] If $y\in U$ is a closed point, then $y$ has exactly $[E:F]_s$ preimages in $E$.
\item[2.] If $y\in U$ is a (not necessarily closed) point, then $v_f(x) = v_f(z)$ for all $z\in f^{-1}(y)\cap E$.
\end{enumerate} Suppose that $y\in U$, and let $x_1,\ldots, x_r$ be the preimages of $y$ lying in $E$. Set $W = \ol{\{y\}}$ and $V_i = \ol{\{x_i\}}$ for each $i$. Again by \hyperref[lem:topdeg]{Lemma~\ref*{lem:topdeg}}, there is some closed point $w\in W\cap U$ such that $f^{-1}(w)\cap E\subseteq V_1\cup\cdots\cup V_r$ and moreover that $w$ has exactly $[V_i:W]_s$ preimages in $V_i$ for each $i$. But then \begin{align*}
\sum_{i=1}^r m_f(x_i) & = \sum_{i=1}^r v_f(x_i)[V_i:W]_s = v_f(z)\sum_{i=1}^r[V_i:W]_s = v_f(z)\times \#f^{-1}(w)\cap E\\
& = v_f(z)[E:F]_s = m_f(z).
\end{align*} This completes the proof.
\end{proof}

\section{Measures on classical varieties}

In order to state and prove an equidistribution theorem for classical varieties, we need to have a language of measures and weak convergence of measures on varieties. Such a language is developed in detail in \cite{me}. In this section we will review the relevant definitions and results, as well as define a pull-pack operation for measures under certain morphisms. The setup for this section is the same as in the previous, namely, we let $f\colon X\to Y$ be a finite surjective flat morphism between irreducible varieties over an algebraically closed field $k$. It is absolutely essential that $X$ and $Y$ be viewed as schemes, allowing for non-closed points; not all results in this section will be true otherwise.

We denote by $\mc{M}(X)$ and $\mc{M}(Y)$ the real vector space of all finite signed Borel measures on $X$ and $Y$ with respect to their Zariski topology. We let $SC(X)$ denote the real vector space of all \emph{semicontinuous functions} functions on $X$, that is, all functions $g\colon X\to \R$ of the form $g = h_1 - h_2$, where $h_i\colon X\to \R$ is a bounded upper semicontinuous function. Similarly we let $SC(Y)$ denote the space of semicontinuous functions on $Y$. We equip both $SC(X)$ and $SC(Y)$ with the supremum norm, making them into normed linear spaces. The following structure theorem is proved in \cite{me}.

\begin{thm}\label{thm:measures} We have the following characterization of measures on $X$. \begin{enumerate}
\item[$1.$] Any measure $\mu\in \mc{M}(X)$ can be written uniquely as an absolutely convergent sum $\mu = \sum_{x\in X} c_x\delta_x$, where $c_x\in \R$ for each $x$, and $\delta_x$ denotes the Dirac probability measure at $x$.
\item[$2.$] Integration induces a duality $\mc{M}(X)\cong SC(X)^*$, analogous to the duality between Radon measures and continuous functions on a compact Hausdorff space.
\end{enumerate} The isomorphism $\mc{M}(X)\cong SC(X)^*$ allows one to pull back both the weak topology (i.e., the topology of pointwise convergence) and the strong topology (i.e., the topology of norm convergence) from $SC(X)^*$ to $\mc{M}(X)$. A sequence $\mu_n\in \mc{M}(X)$ converges in the weak topology to a measure $\mu\in \mc{M}(X)$ if and only if $\mu_n(E)\to \mu(E)$ for each closed set $E\subseteq X$. The collection of Borel probability measures on $X$ is both compact and sequentially compact in the weak topology.
\end{thm}

There is, of course, an analogous theorem for $Y$. We now use the results of the previous section to define a pull-back operator $f^*\colon \mc{M}(Y)\to \mc{M}(X)$. 

\begin{prop}\label{prop:pullback} There is a unique linear operator $f^*\colon \mc{M}(Y)\to \mc{M}(X)$ which satisfies the following two conditions: \begin{enumerate}
\item[$1.$] $f^*$ is continuous in both the weak and strong topologies. 
\item[$2.$] If $y\in Y$, then $f^*\delta_y = \sum_{f(x) = y} m_f(x)\delta_x$.
\end{enumerate} For any measure $\mu\in \mc{M}(Y)$, one has $f_*f^*\mu = [X:_fY]\mu$, where $f_*$ denotes the ordinary push-forward operator on measures. If $\mu$ is positive and has total mass $R$, then $f^*\mu$ is again positive, and has total mass $[X:_f Y]R$.
\end{prop}
\begin{proof} First, assume that such an operator $f^*$ does exist. Let $\mu\in \mc{M}(Y)$, with $\mu = \sum c_y\delta_y$. Let $y_1,y_2,\ldots$ be an enumeration of the points $y\in Y$ for which $c_y\neq 0$; there must be a countable number, as otherwise the sum $\sum c_y\delta_y$ would not converge. Then the measures $\mu_N$ defined by $\mu_N = \sum_{i=1}^N c_{y_i}\delta_{y_i}$ converge strongly to $\mu$ as $N\to \infty$, so by (1) and (2) \[
f^*\mu = \lim_{N\to \infty} f^*\mu_N = \lim_{N\to \infty} \sum_{i=1}^N c_{y_i}f^*\delta_{y_i} = \lim_{N\to \infty}\sum_{i=1}^N \sum_{f(x) = y_i} c_{y_i}m_f(x)\delta_x = \sum_{y\in Y}\sum_{f(x) = y} c_ym_f(x)\delta_x.\] This  derivation shows that $f^*$ is uniquely determined. Moreover, combining this equality with \hyperref[thm:preimage_count]{Theorem~\ref*{thm:preimage_count}} yields the remaining statements in the proposition. It then only remains to show the existence of $f^*$. 

To prove existence, we will exploit the duality $\mc{M}\cong SC^*$ and define $f^*$ to be the adjoint of a certain linear operator $f_*\colon \ol{SC}(X)\to \ol{SC}(Y)$, where $\ol{SC}$ denotes the Banach space closure of $SC$, i.e., the space of all functions which are uniform limits of semicontinuous functions. The operator $f_*$ is given by \[
(f_*\varphi)(y) := \sum_{f(x) = y} m_f(x)\varphi(x).\] First we must check that $f_*$ actually maps $\ol{SC}(X)$ into $\ol{SC}(Y)$. Since the vector space span of all characteristic functions $\chi_E$ of closed sets $E\subseteq X$ is dense in $SC(X)$ by \cite[Lemma 3.4]{me}, it suffices to show that $f_*\chi_E\in SC(Y)$ for any closed set $E\subseteq X$. We will prove this by contradiction. Let $T$ denote the (nonempty) set of closed sets $E\subseteq X$ such that $f_*\chi_E\notin SC(Y)$. Since $X$ is Noetherian, we can find a minimal element $E\in T$. If $E$ is reducible, say $E = E_1\cup E_2$, then \[
f_*\chi_E = f_*\chi_{E_1} + f_*\chi_{E_2} - f_*\chi_{E_1\cap E_2}\] lies in $SC(Y)$ by the minimality of $E$, a contradiction. Therefore $E$ must be irreducible. Note that $f_*\chi_E$ is supported in $F = f(E)$. Furthermore, by \hyperref[prop:nice_neighborhood]{Proposition~\ref*{prop:nice_neighborhood}}, there is a nonempty open subset $U\subseteq F$ such that $f_*\chi_E\equiv m_f(E)$ on $U$. Let $V = F\smallsetminus U$ and $W = f^{-1}(V)\cap E$. One then has $f_*(\chi_E - \chi_W) = m_f(E)\chi_U\in SC(Y)$. By the minimality of $E$, one also has $f_*\chi_W\in SC(Y)$. Thus $f_*\chi_E = f_*\chi_W + m_f(E)\chi_U\in SC(Y)$, a contradiction. We conclude that $\chi_E\in SC(Y)$ for all closed sets $E\subseteq X$.

We have therefore given a well-defined linear map $f_*\colon \ol{SC}(X)\to \ol{SC}(Y)$. We must show that it is bounded. This follows easily from \hyperref[thm:preimage_count]{Theorem~\ref*{thm:preimage_count}}, since for all $y\in Y$ \[
|(f_*\varphi)(y)| = \left|{\sum}_{f(x) = y} m_f(x)\varphi(x)\right|\leq \|\varphi\|{\sum}_{f(x) = y} m_f(x) = \|\varphi\| [X:_f Y].\] It is immediate that the adjoint $f^*\colon \mc{M}(Y)\to \mc{M}(X)$ of $f_*$ is weakly and strongly continuous. It remains to show $f^*$ satisfies condition (2). Let $y\in Y$ and let $E\subseteq X$ be closed. Then \[
(f^*\delta_y)(E) = \int f_*\chi_E\,d\delta_y = (f_*\chi_E)(y) = \sum_{x\in f^{-1}(y)\cap E} m_f(x) = \sum_{f(x) = y} m_f(x)\delta_x(E).\] Therefore $f^*\delta_y$ agrees with $\sum_{f(x) = y} m_f(x)\delta_x$ on closed sets. By \cite[Lemma 2.7]{me}, this is enough to conclude that $f^*\delta_y = \sum_{f(x) = y} m_f(x)\delta_x$.
\end{proof}

\section{Detecting total invariance}

The goal of this section is to show how the generic multiplicity function $v_f$ defined in \S2 can be used to detect totally invariant behavior in certain classes of dynamical systems. In the complex setting this has been done in multiple ways (see, for instance, \cite{MR1863737} and \cite{MR2513537}). In this section we will generalize the approach of \cite{MR2513537} to dynamical systems over arbitrary algebraically closed fields $k$. The dynamical systems we consider here are so-called \emph{polarized dynamical systems}.

\begin{Def} Let $X$ be an irreducible projective variety over $k$, and let $f\colon X\to X$ be an endomorphism of $X$. A \emph{polarization} of $f$ is an ample line bundle $L$ on $X$ such that $f^*L\cong L^d$ for some integer $d\geq 1$. If a polarization $L$ of $f$ is specified, we will say that $f$ is a \emph{polarized} dynamical system, and write $f\colon (X,L)\to (X,L)$ to signify this. The integer $d$ will be called the \emph{algebraic degree} of $f$. Not every $f$ admits a polarization.
\end{Def}

The reason for only considering polarized dynamical systems is that one can, by the following theorem of Fakhruddin \cite{MR1995861}, always embed such a system into projective space, making available certain tools we would not have otherwise. Specifically, the polarization assumption will allow us to make certain intersection theory arguments in \hyperref[prop:degcount]{Propositions~\ref*{prop:degcount}} and \ref{prop:bounddeg} below.

\begin{thm}[Fakhruddin] \label{fak}Let $f\colon (X,L)\to (X,L)$ be a polarized dynamical system of algebraic degree $d$. Then there is an embedding $X\subseteq \pr^r_k$ and a morphism $\Phi\colon \pr^r_k\to \pr^r_k$ with $\Phi^*\mc{O}(1) = \mc{O}(d)$ such that $\Phi(X) = X$ and $\Phi|_X = f$.
\end{thm}

For an overview of ample line bundles and intersection theory, we refer to \cite[Chapter VII]{MR1288998}. Given a very ample line bundle $L$ on $X$ and an irreducible dimension $q$ subvariety $E\subseteq X$, the \emph{degree} of $E$ with respect to $L$ is the intersection $\deg_LE :=(E\cdot L\cdot\cdots \cdot L)$, where here there are $q$ factors of $L$. If $s_1,\ldots, s_q$ are general enough divisors representing $L$, then $\deg_LE$ is exactly the number of points in the intersection $E\cap \mathrm{Div}(s_1)\cap\cdots\cap \mathrm{Div}(s_q)$, counted with multiplicity.

\begin{prop}\label{prop:degcount} Suppose $f\colon (X,L)\to (X,L)$ is a polarized dynamical system of algebraic degree $d$. Let $E\subseteq X$ be an irreducible closed subvariety of dimension $q$ such that $f^n(E) = E$ for some $n\geq 1$. Then $[E:_{f^n} E] = d^{nq}$. In particular, one has  $[X:_f X] = d^{\dim X}$.
\end{prop}
\begin{proof} The projection formula gives that \[
[E:_{f^n} E]\deg_LE = \deg_{f^{n*}L}E = \deg_{L^{d^n}}E = d^{nq}\deg_{L}E.\] Thus $[E:_{f^n}E] = d^{nq}$.  
\end{proof}

\begin{prop}\label{prop:bounddeg} Suppose $f\colon (X,L)\to (X,L)$ is a polarized dynamical system of algebraic degree $d$. Let $W\subseteq X$ be an irreducible subvariety of dimension $q$, and let $F$ be an irreducible subvariety of $f^n(W)$. Let $E_1,\ldots, E_m$ be the components of $f^{-n}(F)$ contained in $W$. Then there is a $C>0$ independent of $n$ and $F$ such that  \[
\sum_{i=1}^m [E_i:_{f^n} F]_s \leq C d^{nq}.\]
\end{prop}
\begin{proof} Replacing $L$ by a power $L^s$, we may assume with no loss of generality that $L$ is very ample. We first prove the inequality in the case where $F = x$ is a closed point of $f^n(W)$. Let $s_1,\ldots, s_q$ be sections of $L$ such that $f^n(W)\cap \mathrm{Div}(s_1)\cap\cdots\cap \mathrm{Div}(s_q)$ is finite and contains $x$. Then one has \[
\#f^{-n}(x)\cap W\leq \#W\cap f^{n*}\mathrm{Div}(s_1)\cap\cdots\cap f^{n*}\mathrm{Div}(s_q)\leq \deg_{f^{n*}L}W = d^{nq}\deg_LW.\] We may therefore take $C = \deg_LW$. To prove the general case, we use \hyperref[lem:topddeg]{Lemma~\ref*{lem:topdeg}} to find a nonempty open subset $U$ of $F$ with the following property: if $x$ is a closed point of $U$, then every element of $f^{-n}(x)\cap W$ lies in exactly one $E_i$, and moreover $\#f^{-n}(x)\cap E_i = [E_i:_{f^n} F]_s$. But then if $x\in U$ is a closed point, $\sum_i [E_i:_{f^n} F]_s = \#f^{-n}(x)\cap W\leq d^{nq}\deg_L W$ by what has been shown for closed points.
\end{proof}

Let us now fix a polarized dynamical system $f\colon (X,L)\to (X,L)$ of algebraic degree $d\geq 2$. We will assume also that $f$ is flat, so that we can apply all the results of \S2. Recall that a set $A\subseteq X$ is said to be \emph{totally invariant} if $f^{-1}(A) = A$. This condition is strictly stronger than ordinary invariance $f(A) = A$. We will say that an irreducible closed set $E\subseteq X$ is part of a  \emph{totally invariant cycle} for $f$ is $E$ is totally invariant for some iterate $f^n$ of $f$. In this case $F := E\cup f(E)\cup\cdots\cup f^{n-1}(E)$ is totally invariant for $f$, and $f$ permutes the irreducible components of $F$ cyclically. As we shall see shortly in \hyperref[thm:totinv]{Theorem~\ref*{thm:totinv}}, total invariance is something that in many cases can be detected by the generic multiplicity function $v_f\colon X\to \N$ defined in \S2. The following functions were first defined and studied by Dinh in the complex setting, see \cite{MR2513537}.

\begin{Def} For each point $y\in X$ and each $n\geq 1$, define \[
v_{-n}(y) := \max_{f^n(x) = y} v_{f^n}(x)\,\,\,\,\,\mbox{ and }\,\,\,\,\, v_-(y) = \lim_{n\to \infty} [v_{-n}(y)]^{1/n}.\] The function $v_-\colon X\to \N$ will be called the \emph{reverse asymptotic multiplicity} function for $f$. It will be convenient to sometimes write $v_-(E)$ in place of $v_-(x)$ when $E = \ol{\{x\}}$. The following theorem shows that $v_-$ is indeed well-defined.
\end{Def}

\begin{thm}[\cite{MR2513537}, see also \cite{me}]\label{thm:usc} For each $y\in X$, the limit $v_-(y)$ exists. Moreover, the reverse asymptotic multiplicity function $v_-\colon X\to \R$ is Zariski upper semicontinuous.
\end{thm}

In order to proceed any further, we will need to make one additional technical assumption about the morphism $f$ to rule out complications resulting from inseparable behavior that arise when working over fields $k$ of positive characteristic.

\begin{Ass}\label{ass:technical} We assume that whenever $E\subseteq X$ is an irreducible closed set which is periodic for $f$, say with period $n$, one has $[E:_{f^n} E]_i = 1$.
\end{Ass}

\begin{prop} If $\mathrm{char}(k) = 0$ or $\mathrm{char}(k) = p\nmid d$, then \hyperref[ass:technical]{Assumption~\ref*{ass:technical}} is automatically satisfied.
\end{prop}
\begin{proof} The proposition is clear when $\mathrm{char}(k) = 0$, so assume $\mathrm{char}(k) =p>0$ and $p\nmid d$. By \hyperref[prop:degcount]{Proposition~\ref*{prop:degcount}}, we have $[E:_{f^n} E] = d^{n\dim(E)}$. Since $p\nmid d$, it follows that the field extension $k(E)/f^{n*}k(E)$ must be separable.
\end{proof}

\begin{thm}\label{thm:totinv} Let $f\colon (X,L)\to (X,L)$ be a polarized dynamical system of algebraic degree $d\geq 2$ which is flat and satisfies \hyperref[ass:technical]{Assumption~\ref*{ass:technical}}. Let $E$ be an irreducible closed subset of codimension $q$ in $X$. Then $v_-(E) \leq d^q$, with equality if and only if $E$ is part of a totally invariant cycle for $f$.
\end{thm}
\begin{proof} Let $m := \dim(X)$. For any $n$-periodic irreducible closed set $F\subseteq X$, let \[
v_+(F) := v_{f^n}(F)^{1/n}.\] It is shown in \cite{me} that $v_-(E) = \max v_+(F)$, where the maximum is taken over all periodic irreducible closed subsets $F\subseteq X$ which contain $E$. Fix a periodic irreducible closed set $F$ containing $E$ such that $v_-(E) = v_+(F)$. Let $n$ be the period of $F$. By \hyperref[prop:multiplicity_relation]{Proposition~\ref*{prop:multiplicity_relation}}, \[
v_+(F) = v_{f^n}(F)^{1/n} = \left(\frac{m_{f^n}(F)}{[F:_{f^n} F]_s}\right)^{1/n} \leq \left(\frac{d^{nm}}{[F:_{f^n} F]_s}\right)^{1/n},\] with equality if and only if $m_{f^n}(F) = d^{nm}$, i.e., if and only if $F$ is totally invariant for $f^n$. By \hyperref[ass:technical]{Assumption~\ref*{ass:technical}}, we have $[F:_{f^n} F]_s = [F:_{f^n} F] = d^{n\dim F}$, so that $v_+(F) \leq d^{\mathrm{codim}(F)}$, with equality if and only if $F$ is part of a totally invariant cycle for $f$. Since $E\subseteq F$, we have that $q\geq \mathrm{codim}(F)$, with equality if and only if $E = F$. It follows that $v_-(E) = v_+(F) \leq d^q$, with equality if and only if $E = F$ is part of a totally invariant cycle for $f$.
\end{proof}

\begin{cor}\label{cor:finiteness} There are finitely many irreducible closed subsets $E\subseteq X$ that are part of a totally invariant cycle for $f$.
\end{cor}
\begin{proof} It is enough to prove there are only finitely many irreducible closed sets $E\subseteq X$ of a fixed codimension $q$ that are part of totally invariant cycles for $f$. Indeed, by \hyperref[thm:totinv]{Theorem~\ref*{thm:totinv}}, the codimension $q$ irreducible closed sets which are part of a totally invariant cycle are precisely the codimension $q$ components of the closed set $\{v_-\geq d^q\}$.
\end{proof}

\begin{Def} The \emph{exceptional set} of $f$ is the set $\ms{E}\subseteq X$ which is the union of all totally invariant proper closed subsets $E\subsetneq X$. If $f$ satisfies \hyperref[ass:technical]{Assumption~\ref*{ass:technical}}, then by \hyperref[cor:finiteness]{Corollary~\ref*{cor:finiteness}} this union is finite, so that $\ms{E}$ is itself a totally invariant proper closed subset of $X$. It is thus the maximal proper totally invariant closed subset.
\end{Def}

It should be noted that \hyperref[ass:technical]{Assumption~\ref*{ass:technical}} cannot be removed from \hyperref[thm:totinv]{Theorem~\ref*{thm:totinv}}. Indeed, if we consider the Frobenius map $f\colon \pr^1_k\to \pr^1_k$ over $k = \ol{\mathbf{F}}_p$, given by $f(z) = z^p$, then one easily checks that \emph{every} point of $\pr^1_k$ is part of a totally invariant cycle for $f$. In particular, there is no maximal proper Zariski closed subset of $\pr^1_k$ which is totally invariant. We therefore see that the techniques developed in the complex setting for detecting total invariance can fail in characteristic $p$ in the presence of inseparable behavior.

In the complex setting, it was shown by Forn\ae ss-Sibony \cite{MR1369137} that generic morphisms $f\colon \pr^r_\C\to \pr^r_\C$ of algebraic degree $d\geq 2$ have $\ms{E} = \varnothing$. We devote the rest of this section to proving that this remains true over any algebraically closed field $k$. Such a statement is made precise as follows. First, recall that  endomorphisms $f\colon \pr^r_k\to \pr^r_k$ of algebraic degree $d$ are naturally parameterized by a certain irreducible affine variety $\mc{H}_d$ over $k$, see for instance \cite[Theorem 1.8]{Silverman:2012fk}. We will show that $\mc{H}_d$ contains a nonempty Zariski open subset of endomorphisms $f$ which have $\ms{E} = \varnothing$. The proof we give is identical in spirit to that of \cite[Theorem 1.3]{MR2468484}, even though the details differ in places.

\begin{prop} Let $v\colon \mc{H}_d\times \pr^r_k\to \R$ be the map, defined on closed points, that is given by $v(f,x) = v_f(x)$. Then $v$ is Zariski upper semicontinuous. 
\end{prop}
\begin{proof} Let $X\subseteq \mc{H}_d\times \pr^r_k\times\pr^r_k$ be the subvariety $X = \{(f,x,y) : f(x) = f(y)\}$, and let $\mc{I}$ denote the ideal sheaf of $\Delta = \{(f,x,y) : x = y\}\subseteq X$. We will denote by $\mc{F}$ the sheaf $\mc{O}_X/\mc{I}^N$, where $N$ is any integer $\geq d^r$. Let $\pi\colon X\to \mc{H}_d\times \pr^r_k$ denote the projection onto the first two coordinates. For any fixed morphism $f\in \mc{H}_d$, one obtains embeddings $i_f\colon \pr^r_k\times_f\pr^r_k\to X$ and $j_f\colon \pr^r_k\to \mc{H}_d\times \pr^r_k$, namely $i_f(x,y) = (f,x,y)$ and $j_f(x) = (f,x)$. Moreover, if $\eta\colon \pr^r_k\times_f\pr^r_k\to \pr^r_k$ is the projection onto the first coordinate, then $\pi\circ i_f = j_f\circ \eta$. We saw in the proof of \hyperref[LJT]{Theorem~\ref*{LJT}} that the fiber dimension of $\eta_*j_f^*\mc{F} = i_f^*\pi_*\mc{F}$ at a point $x\in \pr^r_k$ is exactly $v_f(x)$. On the other hand, this fiber dimension is equal to the fiber dimension of $\pi_*\mc{F}$ at $i_f(x)= (f,x)$. Thus the fiber dimension of $\pi_*\mc{F}$ at $(f,x)$ is exactly $v_f(x)$. Since $\pi_*\mc{F}$ is a coherent sheaf on $\mc{H}_d\times \pr^r_k$, its fiber dimensions are upper semicontinuous.
\end{proof}

\begin{cor}\label{open} For any $a\in \R$, the set of endomorphisms $f\in \mc{H}_d$ such that $v_f(x)<a$ for all $x\in \pr^r_k$ is Zariski open.
\end{cor}
\begin{proof} We will show that the set of $f$ for which there exists a point $x\in \pr^r_k$ with $v_f(x)\geq a$ is Zariski closed. Indeed, this set is the image under the projection map $\pi\colon \mc{H}_d\times \pr^r_k\to \mc{H}_d$ of the closed set $\{(f,x) : v(f,x)\geq a\}$. Since $\pi$ is closed \cite[Theorem I.5.3]{MR1328833}, the corollary follows.
\end{proof}

\begin{prop}\label{multcondition} Let $f\in \mc{H}_d$, and suppose there is an integer $N\geq 1$ such that $v_{f^N}(x)<d^N$ for all closed points $x\in \pr^r_k$. Then $\ms{E} = \varnothing$.
\end{prop}
\begin{proof} Replacing $f$ by an iterate if necessary, we may assume that $N = 1$ and that all irreducible components of $\ms{E}$ are totally invariant. If $E$ is such a component, then $v_-(E) = v_f(E) = d^{\mathrm{codim}(E)}[E:_fE]_i\geq d$, and hence $v_f(x)\geq d$ for all closed points $x\in X$, a contradiction of our assumption that $v_f(x)<d$ for all $x$. Thus $\ms{E} = \varnothing$.
\end{proof}

\begin{thm}\label{genericity} There is an endomorphism $f\in \mc{H}_d$ and a $B>0$ such that $v_{f^n}(x)\leq B$ for all $n\geq 1$ and all $x\in \pr^r_k$. As a consequence, there is a nonempty Zariski open subset of $\mc{H}_d$ consisting of morphisms with empty exceptional set.
\end{thm}
\begin{proof} We begin by proving the theorem in dimension $r = 1$. Suppose first that $d\neq p^m$, where $p = \mathrm{char}(k)>0$. Then there is an $a\in k^\times$,  such that $(a + 1)^d = 1$. The rational function $h\colon \pr^1_k\to \pr^1_k$ given by $h(z) = (z + a)^d/z^d$ then satisfies the condition of the theorem. Indeed, $h$ has two critical points of order $d -1$, namely $0$ and $-a$, and both are strictly preperiodic. It follows that $v_{h^n}(z) \leq d^2$ for all $n\geq 1$ and all $z\in \pr^1_k$. In the case when $\mathrm{char}(k) = p>0$ and  $d = p^m$, a similar argument holds for the rational map $h(z) = (z+1)^d/z^{d-1}$. This map has two critical points $z = 0, -1$ of orders $d-2$ and $d-1$, respectively. Both are strictly preperiodic, so $v_{h^n}(z) \leq d(d-1)$ for all $n\geq 1$ and $z\in \pr^1_k$.

We will now use this $1$-dimensional result to deduce the general case via a construction of Ueda \cite{MR1276837}. Choose a degree $d$ rational map $h\colon \pr^1_k\to \pr^1_k$ and $B>0$ such that $v_{h^n}(z) \leq B$ for all $n\geq 1$ and all $z\in \pr^1_k$. Let $H$ be the endomorphism of the $r$-fold product $\pr^1_k\times\cdots\times \pr^1_k$ given by $H = h\times\cdots\times h$. It is easy to check that for $z = (z_1,\ldots, z_r)\in \pr^1_k\times\cdots\times \pr^1_k$, one has $v_{H^n}(z) = v_{h^n}(z_1)\cdots v_{h^n}(z_r)\leq B^r$. The symmetric group $S_r$ acts on the product $\pr^1_k\times\cdots\times \pr^1_k$ by permuting coordinates, and the quotient $(\pr^1_k\times\cdots\times\pr^1_k)/S_r$ is isomorphic to the projective space $\pr^r_k$. Let $\pi\colon \pr^1_k\times\cdots\times\pr^1_k\to \pr^r_k$ be the quotient morphism. Then $\pi$ is  finite of degree $r!$, and $H$ descends through $\pi$ to an endomorphism $f\in \mc{H}_d$, such that $\pi\circ H = f\circ \pi$. For any $x\in \pr^r_k$ and any $z\in \pi^{-1}(x)$, it follows that \[
v_{f^n}(x) = \frac{v_{f^n}(x)v_\pi(z)}{v_\pi(z)} = \frac{v_{H^n}(z)v_\pi(H^n(z))}{v_\pi(z)}\leq \frac{B^r\cdot r!}{1} = B^rr!.\] This proves the first statement of the theorem. The last statement is now immediate from \hyperref[open]{Corollary~\ref*{open}} and \hyperref[multcondition]{Proposition~\ref*{multcondition}}.
\end{proof}

\section{Equidistribution for classical varieties}

We are now in a position to prove an analogue of the equidistribution of preimages theorem for classical varieties.

\begin{thm}\label{thm:Zariski_equid} Let $k$ be an algebraically closed field, and let $X$ be an irreducible projective variety of dimension $m$ over $k$. Suppose $f\colon (X,L)\to (X,L)$ is a flat polarized endomorphism of $X$ with algebraic degree $d\geq 2$. Assume, furthermore, that $f$ satisfies \hyperref[ass:technical]{Assumption~\ref*{ass:technical}}. Let $x\in X$ be any point, and let $V\subseteq X$ be the smallest totally invariant closed set containing $x$. Assume that $V$ is irreducible, with generic point $y$. Then the sequence $d^{-nm}f^{n*}\delta_x$ of Borel probability measures on $X$ converges weakly to $\delta_y$ as $n\to \infty$.
\end{thm}

Note, the case when $V$ is reducible will be considered in \hyperref[cor:varequid]{Corollary~\ref*{cor:varequid}}.

\begin{proof} Let $\mu_n := d^{-nm}f^{n*}\delta_x$. From \hyperref[thm:measures]{Theorem~\ref*{thm:measures}}, we know that the space of Borel probability measures on $X$ is sequentially compact in the weak topology. It therefore suffices to prove the following: any weakly convergent subsequence $\mu_{n_i}$ of $\mu_n$ converges to $\delta_y$. We therefore fix a weakly convergent subsequence $\mu_{n_i}$, converging to some measure $\mu$. Let $W\subseteq X$ be a minimal closed set with $\mu(W)>0$. If $W$ were reducible, say $W = W_1\cup W_2$, then by the minimality of $W$ we would have $\mu(W)\leq \mu(W_1) + \mu(W_2) = 0 + 0 = 0$, a contradiction. Therefore $W$ is irreducible. One easily sees that $\mu_n(V) = 1$ for all $n$, and hence $\mu(V) = 1$. In particular, $\mu(W\cap V) = \mu(W)>0$, so the minimality of $W$ implies that $W\subseteq V$.

To prove the theorem, it will suffice to show that $W$ is part of a totally invariant cycle for $f$. Indeed, if we can do this, then the minimality of $V$ implies $W = V$. But then $\mu(V) = 1$ and $\mu(Z) = 0$ for all closed $Z\subsetneq V$, implying that $\mu = \delta_y$, as desired. We will prove that $W$ is part of a totally invariant cycle for $f$ by contradiction. Suppose $W$ is not part of a totally invariant cycle for $f$. Using \hyperref[thm:totinv]{Theorem~\ref*{thm:totinv}}, one then has $v_-(W)< d^q$, where $q$ is the codimension of $W$ in $X$. We need the following lemma to proceed.

\begin{lem}\label{lem:good_preimage} There is an integer $I\geq 0$ and a preimage $z\in f^{-n_I}(x)$ such that \begin{enumerate}
\item[$1.$] $z\in W$ and $v_-(z)<d^q$.
\item[$2.$] $\limsup_{i\to \infty} d^{-m(n_i - n_I)}[f^{(n_i - n_I)*}\delta_z](W)>0$.
\end{enumerate}
\end{lem}
\begin{proof} Recall from \hyperref[thm:usc]{Theorem~\ref*{thm:usc}} that the reverse asymptotic multiplicity function $v_-$ is upper semicontinuous. Since $v_-(W) < d^q$, there is a nonempty open subset $U$ of $W$ such that $v_- < d^q$ on $U$. By the minimality of $W$, one has $\mu(W) = \mu(U) = \lim_{i\to \infty} \mu_{n_i}(U)$. We will prove the lemma by contradiction, so suppose no such $z$ and $I$ exist. To simply notation, set \[
R(z,I) := \limsup_{i\to \infty} d^{-m(n_i - n_I)}[f^{(n_i - n_I)*}\delta_z](U)\] whenever $I\geq 0$ is an integer and $z\in f^{-n_I}(x)$. Note that $R(z,I)\leq 1$, and by our contradiction assumption $R(z,I) = 0$ whenever $z\in U$.

\textbf{Claim:} If $I\geq 0$ and $z\in f^{-n_I}(x)$ are such that $R(z,I)\geq c>0$, then there is an integer $J>I$ and a preimage $z'\in f^{-(n_J - n_I)}(z)$ such that $R(z,I)\leq (1 - c/2)R(z',J)$. To prove the claim, let $J>I$ be any integer large enough that $d^{-m(n_J - n_I)}[f^{(n_J - n_I)*}\delta_z](U)\geq c/2$. Suppose that $z_1,\ldots, z_s$ are the elements of $f^{-(n_J - n_I)}(z)\cap U$, and that $z_{s+1},\ldots, z_t$ are the elements of $f^{-(n_J - n_I)}(z)$ lying outside $U$. Then \[
R(z,I)\leq \sum_{i=1}^t \frac{m_{f^{n_J - n_I}}(z_i)}{d^{m(n_J - n_i)}}R(z_i,J) = \sum_{i=s+1}^t  \frac{m_{f^{n_J - n_I}}(z_i)}{d^{m(n_J - n_i)}}R(z_i,J)\] since $R(z_i, J) = 0$ for all $i\leq s$. One then has the easy upper bound \begin{align*}
R(z,I) & \leq \max\{R(z_{s+1}, J),\ldots, R(z_t,J)\}\sum_{i=s+1}^t \frac{m_{f^{n_J - n_I}}(z_i)}{d^{m(n_J - n_I)}}\\
& = \max\{R(z_{s+1},J),\ldots, R(z_t,J)\}d^{-m(n_J - n_I)}[f^{(n_J - n_I)*}\delta_z](X\smallsetminus U).
\end{align*} By our choice of $J$, it follows that \[
R(z,I)\leq (1 - c/2)\max\{R(z_{s+1},J),\ldots, R(z_t,J)\},\] proving the claim.

Let $c = \mu(W)$. By definition, $R(x,0) = c$, so the claim yields an integer $I_1>0$ and a preimage $z_1\in f^{-n_{I_1}}(x)$ such that $\mu(W) = c = R(x,0) \leq (1 - c/2)R(z_1,I_1)$. In particular, \[
R(z_1,I_1)\geq \frac{c}{1 - c/2} > c.\] We can thus apply the claim again to find an integer $I_2>I_1$ and a  $z_2\in f^{-(n_{I_2} - n_{I_1})}(z_1)$ such that $R(z_1,I_1)\leq (1 - c/2)R(z_2,I_2)$. Thus $\mu(W) = R(x,0)\leq (1 - c/2)^2R(z_2,I_2)$. Continuing in this fashion, we construct sequences $I_j$ and $z_j$ such that \[
\mu(W)\leq (1 - c/2)^j R(z_j,I_j) \leq (1 - c/2)^j \to 0.\] This contradicts the assumption that $\mu(W)>0$, and completes the proof.
\end{proof}

We now continue with the proof of \hyperref[thm:Zariski_equid]{Theorem~\ref*{thm:Zariski_equid}}. Let $I$ and $z$ be as in the statement of \hyperref[lem:good_preimage]{Lemma~\ref*{lem:good_preimage}}. Let $\Delta\in \R$ be such that $v_-(z) < \Delta < d^q$. Passing to a subsequence if necessary, we may assume that the limit \[\tag{$*$}
c:= \lim_{i\to \infty} d^{-m(n_i-n_I)}[f^{(n_i - n_I)*}\delta_z](W)\] exists and is positive. For each $i\geq I$, let $z_1^i,\ldots, z_{s_i}^i$ denote the elements of $f^{-(n_i-n_I)}(z)$ which lie in $W$. Then the right hand side of ($*$) is \[
\lim_{i\to \infty} d^{-m(n_i - n_I)}\sum_{j=1}^{s_i} m_{f^{n_i - n_I}}(z_j^i) = \lim_{i\to \infty} d^{-m(n_i - n_I)}\sum_{j=1}^{s_i} v_{f^{n_i - n_I}}(z_j^i) [E_j^i:_{f^{n_i - n_I}} E]_s,\] where $E_j^i = \ol{\{z_j^i\}}$ and $E = \ol{\{z\}}$. Since $v_-(z) < \Delta$, we have $v_{f^{n_i - n_I}}(z_j^i)\leq \Delta^{n_i - n_I}$ for every $j$ whenever $i$ is sufficiently large. Also, by \hyperref[prop:bounddeg]{Proposition~\ref*{prop:bounddeg}}, we have \[
\sum_{j=1}^s [E_j^i :_{f^{n_i - n_I}} E]_s \leq Cd^{(n_i - n_I)\dim(W)}\] for some $C>0$ independent of $i$. Combining these inequalities, we see that \[
c \leq \limsup_{i\to \infty} d^{-m(n_i - n_I)}\Delta^{n_i - n_I}Cd^{(n_i - n_I)\dim(W)} = C\limsup_{i\to \infty}\,  (d^{-q}\Delta)^{n_i - n_I} = 0,\] where here the last equality results from the fact that $\Delta<d^q$. This is a contradiction of the fact that $c>0$. Therefore $W$ is totally invariant, completing the proof.
\end{proof}

From this equidistribution theorem we derive a couple of easy variants.

\begin{cor}\label{cor:varequid} Let $f\colon X\to X$ be as in \hyperref[thm:Zariski_equid]{Theorem~\ref*{thm:Zariski_equid}}. Let $x\in X$ be any point, and let $V$ be the smallest totally invariant closed subset of $X$ containing $x$. Let $V = V_0\cup\cdots\cup V_{s-1}$ be the irreducible decomposition of $V$, and let $y_i$ be the generic point of $V_i$ for each $i$. Then, after relabeling the $V_i$ if necessary, one has for each $i = 0,\ldots, s-1$ that $d^{-m(i+sn)}f^{(i+sn)*}\delta_x\to \delta_{y_i}$ weakly as $n\to \infty$.
\end{cor}
\begin{proof} Without loss of generality, we may assume that $x\in V_0$, and that $f(V_i) = V_{i-1}$, the indices taken modulo $s$. Note, in particular, that $d^{-m}f^*\delta_{y_i} = \delta_{y_{i+1}}$. The set $V_0$ is totally invariant for the composition $f^s$, and is in fact the minimal $f^s$-totally invariant closed set containing $x$. Thus by \hyperref[thm:Zariski_equid]{Theorem~\ref*{thm:Zariski_equid}}, $d^{-msn}f^{sn*}\delta_x\to \delta_{y_0}$ weakly as $n\to \infty$. We know that the pull-back operator $f^*\colon\mc{M}(X)\to \mc{M}(X)$ is weakly continuous by \hyperref[prop:pullback]{Proposition~\ref*{prop:pullback}}, and thus for any $i = 0,\ldots, s-1$ we see that \[
d^{-m(i+sn)}f^{(i+sn)*}\delta_x = d^{-mi}f^{i*}[d^{-msn}f^{sn*}\delta_x]\to d^{-mi}f^{i*}\delta_{y_0} = \delta_{y_i},\] as desired.
\end{proof}

\begin{cor}\label{cor:classical_equid} Let $f\colon X\to X$ be as in \hyperref[thm:Zariski_equid]{Theorem~\ref*{thm:Zariski_equid}}.  Let $\mu$ be a Borel probability measure on $X$ that gives no mass to the exceptional set $\ms{E}$ of $f$. Then $d^{-mn}f^{n*}\mu\to \delta_y$ weakly as $n\to \infty$, where $y$ is the generic point of $X$.
\end{cor}
\begin{proof} Let $\mu = \sum_{x\in X} c_x\delta_x$. For each $x\in X$ such that $c_x\neq 0$, one has $x\notin \ms{E}$, as otherwise $\mu(\ms{E})\neq 0$. For such $x$, it follows from \hyperref[thm:Zariski_equid]{Theorem~\ref*{thm:Zariski_equid}} that the preimages of $x$ equidistribute to $\delta_y$. Let $x_1,x_2,\ldots$ be an enumeration of those $x\in X$ with $c_x\neq 0$. For each $N\geq 1$, let $\mu_N = \sum_{i=1}^N c_{x_i}\delta_{x_i}$ and $\nu_N = \sum_{i>N} c_{x_i}\delta_{x_i}$. Let $\eps>0$ be given, and choose $N$ large enough so that $\nu_N(X)<\eps$. Then for any closed set $E\subseteq X$, one has \[
|d^{-mn}(f^{n*}\mu)(E) - \delta_y(E)| \leq |d^{-mn}(f^{n*}\mu_N)(E) - \delta_y(E)| + \eps\] for every $n\geq 1$. When $n$ is sufficiently large, however, \hyperref[thm:Zariski_equid]{Theorem~\ref*{thm:Zariski_equid}} gives that \[|d^{-mn}(f^{n*}\mu_N)(E) - \mu_N(X)\delta_y(E)| \leq \eps.\] Combining this with the previous inequality yields \[
|d^{-mn}(f^{n*}\mu)(E) - \delta_y(E)| \leq (1 - \mu_N(X))\delta_y(E) + 2\eps \leq 3\eps.\] Therefore $d^{-mn}f^{n*}\mu \to \delta_y$ as $n\to \infty$.
\end{proof}

\section{Berkovich analytic spaces}

Having proved the equidistribution theorem for classical varieties, we now move on to the nonarchimedean setting. Fix an algebraically closed complete nonarchimedean field $K$. We do not assume that the absolute value on $K$ is nontrivial. In this section we briefly review the Berkovich analytification of varieties over $K$, and discuss multiplicities for finite morphisms between analytic varieties. The main references for the material in this section are the works of Berkovich  \cite{MR1070709} and \cite{MR1259429}.

We begin by discussing the Berkovich analytification of a variety $X$ over $K$. Suppose first that $X$ is affine, with coordinate ring $K[X]$. Then the Berkovich analytification of $X$ is, as a set, defined to be the collection $X^\an$ of all multiplicative seminorms $K[X]\to \R$ which extend the given absolute value on $K$. We will denote points in $X^\an$ by letters such as $\bk{x}$ and $\bk{y}$. By definition these are seminorms on $K[X]$; the value of $\bk{x}$ on a function $\varphi\in K[X]$ is typically denoted $|\varphi(\bk{x})|$. We equip $X^\an$ with the weakest topology for which each of the evaluation maps $\bk{x}\mapsto |\varphi(\bk{x})|$ for $\varphi\in K[X]$ are continuous. In this topology, $X^\an$ is locally compact, Hausdorff, and locally path connected.

The closed points of $X$ naturally embed into $X^\an$ in the following way. If $x\in X$ is a closed point, we define a seminorm $K[X]\to \R$ associated to $x$ by $\varphi\in K[X]\mapsto |\varphi(x)|$, where the absolute value here is the given absolute value on $K$. If the absolute value on $K$ is trivial, then the scheme-theoretic points of $X$ also embed into $X^\an$. Indeed, if $x$ is any scheme theoretic point of $X$ corresponding to the prime ideal $\mf{p}$ of $K[X]$, we can define a corresponding seminorm $K[X]\to \R$ by \[
\varphi\in K[X]\mapsto \begin{cases} 1 & \varphi\notin \mf{p}.\\ 0 & \varphi\in \mf{p}.\end{cases}\] In either the trivially or non-trivially valued case, we will call seminorms of this form \emph{classical points}. Classical points are dense in $X^\an$ when the absolute value on $K$ is nontrivial.

There is a natural map $\pi\colon X^\an\to X$, where $X$ here is viewed as a scheme, allowing for non-closed points. The map $\pi$ sends a seminorm $\bk{x}\in X^\an$ to its \emph{kernel} \[
\pi(\bk{x}) := \{\varphi\in K[X] : |\varphi(\bk{x})| = 0\},\] which is easily seen to be a prime ideal. If $\bk{x}$ is a classical point of $X^\an$ corresponding to $x\in X$, then $\pi(\bk{x}) = x$. The map $\pi$ is continuous when $X$ is given its Zariski topology.

Suppose now that $f\colon X\to Y$ is a morphism of affine varieties, with corresponding homomorphism $f^*\colon K[Y]\to K[X]$ of coordinate rings. We can define a map $f^\an\colon X^\an\to Y^\an$ by sending a seminorm $\bk{x}\in X^\an$ to the seminorm $f^\an(\bk{x})$ defined by $|\varphi(f^\an(\bk{x}))| = |(f^*\varphi)(\bk{x})|$ for all $\varphi\in K[Y]$. The map $f^\an$ is continuous, and agrees with the map $f\colon X\to Y$ on classical points. For this reason, we will abuse notation and denote $f^\an$ simply by $f$.

Now assume that $X$ is any variety over $K$, not necessarily affine. One defines the Berkovich analytification $X^\an$ of $X$ as follows. Choose a finite open cover of $X$ by affines $U_1,\ldots, U_r$. One obtains the space $X^\an$ by gluing together the analytifications $\pi_i\colon U_i^\an\to U_i$. Namely, we identify seminorms $\bk{x}\in U_i^\an$ and $\bk{y}\in U_j^\an$ if $\pi_i(\bk{x})\in U_i\cap U_j$ and $\pi_j(\bk{y})\in U_i\cap U_j$ and if $\bk{x}$ and $\bk{y}$ give the same seminorm on $\mc{O}_{U_i}(U_i\cap U_j)\cong \mc{O}_{U_j}(U_i\cap U_j)$. The space $X^\an$ constructed in this fashion is locally compact, Hausdorff, and locally path connected. Moreover, if $X$ is an irreducible projective variety, then $X^\an$ is compact and connected. The $\pi_i\colon U_i^\an\to U_i$ glue together to a map $\pi\colon X^\an\to X$, which is continuous when $X$ is given its Zariski topology.

As before, closed points of $X$ naturally embed into $X^\an$, and if $K$ is equipped with the trivial absolute value, so do the scheme-theoretic points of $X$. These points in $X^\an$ are the \emph{classical points}. If $f\colon X\to Y$ is a morphism of varieties over $K$, there is an induced continuous map $f\colon X^\an\to Y^\an$, which agrees with $f\colon X\to Y$ on classical points.

While we do not go into details here (see \cite{MR1070709}), one can define a sheaf of rings on $X^\an$, called the \emph{structure sheaf} of $X^\an$. We will denote this sheaf by $\ms{O}_X$ to distinguish it from the classical structure sheaf $\mc{O}_X$ on the variety $X$. Equipped with this sheaf $\ms{O}_X$, the analytification $X^\an$ is a locally ringed space, and $\pi\colon X^\an\to X$ is a morphism of locally ringed spaces. The Berkovich analytification $\pi\colon X^\an\to X$ enjoys GAGA results analogous to those in the classical complex setting (see \cite[\S\S3.4-3.5]{MR1070709}). We will now use these GAGA results to discuss multiplicities associated to finite morphisms of analytic varieties, as in \S2.

\begin{Def} Let $X$ and $Y$ be irreducible varieties over $K$, and let $f\colon X\to Y$ be a finite surjective morphism. Let $\bk{x}\in X^\an$ and $\bk{y} = f(\bk{x})$. Then the \emph{multiplicity} of $f$ at $\bk{x}$ is \[
m_f(\bk{x}) := \dim_{\kappa(\bk{y})}(\ms{O}_{X,\bk{x}}/\mf{m}_\bk{y}\ms{O}_{X,\bk{x}}),\] where as usual $\ms{O}_{X,\bk{x}}$ is viewed as an $\ms{O}_{Y,\bk{y}}$-module via $f$. 
\end{Def}

We now want to compare these multiplicities with those previously defined for the classical morphism $f\colon X\to Y$. To do this comparison, we use the commutative diagram 

\bigskip

\[\begin{psmatrix}[colsep = 1 in, rowsep = .5in]
X^\an & Y^\an\\
X & Y
\psset{arrows=->, nodesep = 3pt}
\ncline{1,1}{1,2}^f
\ncline{2,1}{2,2}^f
\ncline{1,1}{2,1}
\tlput{\pi_X}
\ncline{1,2}{2,2}
\trput{\pi_Y}
\end{psmatrix}\]

\bigskip

\noindent Specifically, we note that if $\bk{y}\in Y^\an$ and $y = \pi_Y(\bk{y})$, then for any $\bk{x}\in f^{-1}(\bk{y})$ we must have that $\pi_X(\bk{x})\in f^{-1}(y)$.

\begin{prop}\label{prop:kernel_multiplicities} Let $X$ and $Y$ be irreducible varieties over $K$, and suppose $f\colon X\to Y$ is a finite surjective morphism. Let $x\in X$ and $f(x) = y$. Let $\bk{y}\in Y^\an$ be such that $\pi_Y(\bk{y}) = y$, and let $\bk{x}_1,\ldots, \bk{x}_r$ be those $f$-preimages of $\bk{y}$ such that $\pi_X(\bk{x}_i) = x$. Then \[
m_f(x) = \sum_{i=1}^r m_f(\bk{x}_i).\] In particular, if $\bk{x}$ is a classical point of $X^\an$ corresponding to $x\in X$, then $m_f(\bk{x}) = m_f(x)$.
\end{prop}
\begin{proof} As the statement is local, we may assume with no loss of generality that $X$ and $Y$ are affine. Let $\bk{x}_{r+1},\ldots,\bk{x}_s$ be those preimages of $\bk{y}$ with $\pi_X(\bk{x}_i)\neq x$. Using Proposition 2.6.10 of \cite{MR1259429}, one has the isomorphism \[
\ms{O}_{Y,\bk{y}}\otimes_{\mc{O}_{Y,y}} \mc{O}_{X,x}\cong \prod_{i=1}^r\ms{O}_{X, \bk{x}_i}\,\times\,\prod_{i=r+1}^s (\ms{O}_{X,\bk{x}_i})_{\mf{p}_x},\] where $\mf{p}_x$ is the prime ideal in the coordinate ring $K[X]$ of $X$ which corresponds to $x$. If we then tensor this expression over $\ms{O}_{Y,\bk{y}}$ with the residue field $\kappa(\bk{y})$, we see that \[
\kappa(\bk{y})\otimes_{\kappa(y)} (\mc{O}_{X,x}/\mf{m}_y\mc{O}_{X,x})\cong \prod_{i=1}^r (\ms{O}_{X,\bk{x}_i}/\mf{m}_\bk{y}\ms{O}_{X,\bk{x}_i}).\] The $\kappa(\bk{y})$-dimension of the left hand side of this expression is $m_f(x)$, while the $\kappa(\bk{y})$-dimension of the right hand side is $\sum_{i=1}^r m_f(\bk{x}_i)$. This completes the proof.
\end{proof}

\begin{cor} Suppose $f\colon X\to Y$ is a finite surjective flat morphism between irreducible varieties over $K$. Then every point $\bk{y}\in Y^\an$ has $[X:_fY]$ preimages when counted according to their multiplicity. That is, $[X:_fY] = \sum_{f(\bk{x}) = \bk{y}} m_f(\bk{x})$.
\end{cor}
\begin{proof} Let $y = \pi_Y(\bk{y})$. From \hyperref[prop:kernel_multiplicities]{Proposition~\ref*{prop:kernel_multiplicities}} we know that $\sum_{f(\bk{x}) = \bk{y}} m_f(\bk{x}) = \sum_{f(x) = y} m_f(x)$. The corollary then follows from \hyperref[thm:preimage_count]{Theorem~\ref*{thm:preimage_count}}.
\end{proof}

\begin{prop}\label{prop:FRL} Suppose $f\colon X\to Y$ is a finite surjective flat morphism between irreducible varieties over $K$.  Let $V$ be an affinoid domain in $Y^\an$, and let $U = f^{-1}(V)$. Suppose that $U$ has connected components $U_1,\ldots, U_s$. Then there exist integers $n_1,\ldots, n_s\geq 1$ such that every point $\bk{y}\in V$ has exactly $n_i$ preimages in $U_i$, counted according to their multiplicity.
\end{prop}
\begin{proof} This statement is a higher dimensional analogue of \cite[Proposition 2.1]{MR2578470}.  First note that $U$ is itself an affinoid domain by \cite[Proposition 3.1.7]{MR1070709}, as are the $U_i$ by \cite[Corollary 2.2.7]{MR1070709}. If $\mc{A}_V\to \mc{A}_U\cong \mc{A}_{U_1}\times\cdots\times \mc{A}_{U_s}$ is the corresponding map of affinoid algebras, then $\mc{A}_U$ is a finite Banach $\mc{A}_V$-module since $f$ is finite. It follows immediately that each $\mc{A}_{U_i}$ is a finite Banach $\mc{A}_V$-module via the composite $\mc{A}_V\to \mc{A}_U\to \mc{A}_{U_i}$. Therefore $f|_{U_i}\colon U_i\to V$ is a finite map of $K$-analytic spaces. Since $f$ is flat, so is its analytification $f^\an$ by the GAGA principles in \S3.4 and \S3.5 of  \cite{MR1070709}. It follows that $f|_{U_i}$ is flat for each $i$. Thus  $f_*(\ms{O}_X|_{U_i})$ is a coherent, locally free $\ms{O}_Y|_V$-module of some rank $n_i$. If $\bk{y}\in V$, then \[
f_*(\ms{O}_X|_{U_i})_\bk{y} \cong \bigoplus_{\bk{x}\in f^{-1}(\bk{y})\cap U_i} \ms{O}_{X,\bk{x}},\] and therefore \[
n_i = \dim_{\kappa(\bk{y})} (\kappa(\bk{y})\otimes_{\ms{O}_{Y,\bk{y}}} f_*(\ms{O}_X|_{U_i})_\bk{y}) = \sum_{\bk{x}\in f^{-1}(\bk{y})\cap U_i} \dim_{\kappa(\bk{y})} (\ms{O}_{X,\bk{x}}/\mf{m}_\bk{y}\ms{O}_{X,\bk{x}}) = \sum_{\bk{x}\in f^{-1}(\bk{y})\cap U_i} m_f(\bk{x}).\] This completes the proof.
\end{proof}

Using these results, we are now able to define a pull-back operator on Radon measures, analogous to the pull-back defined in \S3. As was done in that section, we will define the pull-back operator as the adjoint of a push-forward operator on functions. 

\begin{Def} Suppose $f\colon X\to Y$ is a finite surjective flat morphism between irreducible varieties over $K$. If $\varphi\in C^0(X^\an)$, we define $f_*\varphi\colon Y^\an\to \R$ by \[
(f_*\varphi)(\bk{y}) := \sum_{f(\bk{x}) = \bk{y}} m_f(\bk{x})\varphi(\bk{x}).\]
\end{Def}

\begin{prop} Suppose $f\colon X\to Y$ is a finite surjective flat morphism between irreducible varieties over $K$. Then the push-forward $f_*$ defines a linear map $C^0(X^\an)\to C^0(Y^\an)$. If we assume, in addition, that $X$ and $Y$ are projective, so that $X^\an$ and $Y^\an$ are compact,  then $f_*$ is a bounded linear operator between Banach spaces. 
\end{prop}
\begin{proof} Let $\varphi\in C^0(X^\an)$, and let $\bk{y}\in Y^\an$. Let $V$ be a small enough affinoid neighborhood of $\bk{y}$ such that $f^{-1}(V)$ is a disjoint union of components $U_1,\ldots, U_s$, each containing exactly one preimage $\bk{x}_i\in U_i$ of $\bk{y}$. By shrinking $V$ if necessary, we may assume that the variation of $\varphi$ on $U_i$ is at most $\eps$ for each $i$. According to \hyperref[prop:FRL]{Proposition~\ref*{prop:FRL}}, if $\bk{y}'\in V$, then $\bk{y}'$ has exactly $m_f(\bk{x}_i)$ preimages in $U_i$ when counted according to their multiplicity. Thus \begin{align*}
|(f_*\varphi)(\bk{y}') - (f_*\varphi)(\bk{y})|& \leq \sum_{i=1}^s \left|m_f(\bk{x}_i)\varphi(\bk{x}_i) - {\sum}_{\bk{x}\in f^{-1}(\bk{y}')\cap U_i} m_f(\bk{x})\varphi(\bk{x})\right|\\
&\leq \sum_{i=1}^s \eps m_f(\bk{x}_i) = [X:_f Y]\eps.
\end{align*} This proves that $f_*\varphi$ is continuous.  In the case where $X^\an$ and $Y^\an$ are compact, and thus that $C^0(X^\an)$ and $C^0(Y^\an)$ are Banach spaces, the fact that $f_*$ is bounded is immediate from the easy estimate $|(f_*\varphi)(\bk{y})| \leq \|\varphi\|\sum_{f(\bk{x}) = \bk{y}} m_f(\bk{x}) = [X:_f Y]\|\varphi\|$.
\end{proof}

\begin{Def} Let $f\colon X\to Y$ be a finite surjective flat morphism between irreducible projective varieties over $K$. Let $\mc{M}(X^\an)$ and $\mc{M}(Y^\an)$ denote the space of Radon measures on $X^\an$ and $Y^\an$, respectively. We define the pull-back operator $f^*\colon \mc{M}(Y^\an)\to \mc{M}(X^\an)$ to be the adjoint of $f_*\colon C^0(X^\an)\to C^0(Y^\an)$.
\end{Def}

The following properties of the pull-back $f^*\colon \mc{M}(Y^\an)\to \mc{M}(X^\an)$ are now obvious from the definitions: \begin{enumerate}
\item[$1.$] If $\mu$ is a positive Radon measure on $Y^\an$, then $f^*\mu$ is positive as well. Moreover, if the total mass of $\mu$ is $R$, then $f^*\mu$ has total mass $[X:_f Y]R$.
\item[$2.$] If $\mu$ is any Radon measure on $Y^\an$, then $f_*f^*\mu = [X:_fY] \mu$, where $f_*$ denotes the usual push-forward operation on measures.
\item[$3.$] If $\bk{y}\in Y^\an$ and $\delta_\bk{y}$ is the Dirac probability measure at $\bk{y}$, then $f^*\delta_\bk{y} = \sum_{f(\bk{x}) = \bk{y}} m_f(\bk{x})\delta_\bk{x}$.
\end{enumerate}

\section{Maps of good reduction}

Before being able to prove the equidistribution theorem for maps of good reduction, we first must say  what we mean by a map of good reduction. In section we discuss the notion of reduction for analytic varieties. Recall that we are working with analytic varieties over an algebraically closed complete nonarchimedean field $K$, possibly with trivial absolute value. We denote by $K^\circ$ the valuation ring of $K$, by $\mf{m}_K$ the maximal ideal in $K^\circ$, and by $k$ the residue field on $K$. 

\begin{Def} Let $X$ be an irreducible projective variety over $K$. A \emph{model} of $X$ is a flat, projective scheme $\mc{X}$ over $\Spec K^\circ$ with a specified isomorphism between $X$ and the generic fiber $\mc{X}_K$ of $\mc{X}$.
\end{Def}

In the case when $K$ is equipped with the trivial absolute value, any model $\mc{X}$ of $X$ is simply a variety over $K$ that is isomorphic to $X$, and thus we lose no generality by taking $\mc{X} = X$. When $K$ is equipped with a nontrivial absolute value, there is in general no canonical model of $X$, but some model $\mc{X}$ always exists.

Given a model $\mc{X}$ of $X$, we are able to define a \emph{reduction map} $\bk{red}\colon X^\an\to \mc{X}_k$, where here $\mc{X}_k$ denotes the special fiber of $\mc{X}$. The special fiber $\mc{X}_k$ is a projective variety over $k$, all of whose components have the same dimension as $X$. Let $\bk{x}\in X^\an$, and let $x := \pi(\bk{x}) \in X = \mc{X}_K$. Then $\bk{x}$ defines an absolute value on the residue field $\kappa(x)$ of $\mc{X}_K$ at $x$. Since $\mc{X}$ is projective and hence proper over $\Spec K^\circ$, the valuative criterion of properness gives that the $K$-morphism $\Spec \kappa(x)\to \mc{X}_K$ extends uniquely to a $K^\circ$-morphism $\Spec \kappa(x)^\circ\to \mc{X}$. The special fiber of this morphism is a morphism $\Spec \wt{\kappa}(x)\to \mc{X}_k$, corresponding to a point $\xi$ of $\mc{X}_k$. The reduction of $\bk{x}$ is defined to be $\bk{red}(\bk{x}) = \xi$. It should be noted that the point $\xi$ is also sometimes called the \emph{center} of the valuation associated to $\bk{x}$; we will not use this terminology here.

The reduction map $\bk{red}\colon X^\an\to \mc{X}_k$ is anticontinuous in the sense that the inverse image of a Zariski open set of $\mc{X}_k$ is closed in $X^\an$. It is always surjective, and every generic point of $\mc{X}_k$ has exactly one preimage under $\bk{red}$.

If $K$ is equipped with the trivial absolute value, then $\mc{X}_k = X$, so one obtains a canonical reduction map $\bk{red}\colon X^\an\to X$. The two maps $\pi,\bk{red}\colon X^\an\to X$ are different; indeed, $\bk{red}(\bk{x})$ specializes $\pi(\bk{x})$ for every $\bk{x}\in X^\an$, and one has $\bk{red}(\bk{x}) = \pi(\bk{x})$ if and only if $\bk{x}$ is a classical point of $X^\an$. The unique point $\bk{x}\in X^\an$ whose reduction is the generic point of $X$ is called the \emph{Gauss point} of $X^\an$. It is the classical point corresponding to the trivial absolute value on $K(X)$.

Suppose $X$ and $Y$ are irreducible projective varieties over $K$, and let $\mc{X}$ and $\mc{Y}$ be models of $X$ and $Y$, respectively. Suppose that $F\colon \mc{X}\to \mc{Y}$ is a finite flat $K^\circ$-morphism. Then the generic and special fibers $F_K\colon X\to Y$ and $F_k\colon \mc{X}_k\to \mc{Y}_k$ are finite flat morphisms which are compatible with reduction in the sense that the following diagram commutes.

\bigskip

\[\begin{psmatrix}[colsep = 1 in, rowsep = .5in]
X^\an & Y^\an\\ 
\mc{X}_k & \mc{Y}_k
\psset{arrows=->, nodesep = 3pt}
\ncline{1,1}{1,2}^{F_K}
\ncline{2,1}{2,2}^{F_k}
\ncline{1,1}{2,1}
\tlput{\bk{red}}
\ncline{1,2}{2,2}
\trput{\bk{red}}
\end{psmatrix}\]

\bigskip\smallskip

\noindent If $\mc{X}_k$ and $\mc{Y}_k$ are irreducible, then $[\mc{X}_k :_{F_k} \mc{Y}_k] = [X:_{F_K} Y]$. This compatibility with reduction is the spirit of what we wish to call ``good reduction."

\begin{Def} Let $X$ be an irreducible projective variety over $K$, and let $f\colon X\to X$ be a finite surjective flat morphism. Suppose that there exists a model $\mc{X}$ of $X$ with irreducible special fiber, and a finite flat $K^\circ$-morphism $F\colon \mc{X}\to \mc{X}$ such that $f = F_K$. Then we say that $f$ has \emph{good reduction} with respect to $F$. The map $F_k$ is called the \emph{reduction} of $f$.
\end{Def}

If $K$ is equipped with the trivial absolute value, then every finite surjective flat morphism $f\colon X\to X$  has good reduction simply by taking $F = f$. On the other hand, if $K$ is equipped with a nontrivial absolute value, the notion of good reduction is quite restrictive.

In the case when $X = \pr^r_K$, one sometimes says that a morphism $f\colon \pr^r_K\to \pr^r_K$ has good reduction without making mention of any specific model, as we did in the introduction. Here it is implied that $f$ is induced by a morphism $F\colon \mc{X}\to \mc{X}$, where $\mc{X}$ is the model $\pr^r_{K^\circ}$ of $\pr^r_K$. When $X = \pr^1_K$ the situation is rather simpler, in that every morphism $f\colon \pr^1_K\to \pr^1_K$ of good reduction with respect to some morphism $\mc{X}\to \mc{X}$ is, possibly after conjugating $f$ by an automorphism of $\pr^1_K$, induced from a morphism $\pr^1_{K^\circ}\to \pr^1_{K^\circ}$. One sometimes calls $f\colon \pr^1_K\to \pr^1_K$ a map \emph{potentially good reduction} if it is induced by some morphism $\mc{X}\to \mc{X}$, and a map of \emph{good reduction} if it is induced by a morphism $\pr^1_{K^\circ}\to \pr^1_{K^\circ}$.

Before moving on to the proof of the equidistribution theorem for maps of good reduction, we need a proposition relating the multiplicities of a morphism $f\colon X\to X$ of good reduction to the multiplicities of its reduction.

\begin{prop}\label{prop:reduced_mults} Let $X$ be an irreducible projective variety over $K$, and let $f\colon X\to X$ be a morphism which has good reduction with respect to a morphism $F\colon \mc{X}\to \mc{X}$. Let $\bk{y}\in X^\an$, and let $y = \bk{red}(\bk{y})$. Fix any $x\in \mc{X}_k$ with $F_k(x) = y$, and let $\bk{x}_1,\ldots, \bk{x}_r$ be those $f$-preimages of $\bk{y}$ with $\bk{red}(\bk{x}_i) = x$. Then \[
m_{F_k}(x) = \sum_{i=1}^r m_f(\bk{x}_i).\]
\end{prop}
\begin{proof} Let $\hat{\mc{O}}_{\mc{X},y}$ and $\hat{\mc{O}}_{\mc{X},x}$ be the completions of $\mc{O}_{\mc{X},y}$ and $\mc{O}_{\mc{X},x}$ with respect to their maximal ideals. Then $\hat{\mc{O}}_{\mc{X},x}$ is a finite free $\hat{\mc{O}}_{\mc{X},y}$-module via $F$, say of rank $R$. Because $y = \bk{red}(\bk{y})$, we have a natural $K^\circ$-homomorphism $\hat{\mc{O}}_{\mc{X},y}\to \ms{H}(\bk{y})^\circ$, where $\ms{H}(\bk{y})$ is the completed residue field at $\bk{y}$. This homomorphism allows us to consider the tensor products \[
\hat{\mc{O}}_{\mc{X},x}\otimes_{\hat{\mc{O}}_{\mc{X},y}}\wt{\ms{H}}(\bk{y})\,\,\,\,\,\mbox{ and }\,\,\,\,\, \hat{\mc{O}}_{\mc{X},x}\otimes_{\hat{\mc{O}}_{\mc{X},\eta}}\ms{H}(\bk{y}),\] which are then vector spaces of dimension $R$ over $\wt{\ms{H}}(\bk{y})$ and $\ms{H}(\bk{y})$, respectively. Since \[
\hat{\mc{O}}_{\mc{X},x}\otimes_{\hat{\mc{O}}_{\mc{X},y}}\wt{\ms{H}}(\bk{y})\cong (\mc{O}_{\mc{X}_k,x}/\mf{m}_y\mc{O}_{\mc{X}_k,x})\otimes_{\kappa(y)} \wt{\ms{H}}(\bk{y}),\] one has $m_{F_k}(x) = R$. On the other hand, \[
\hat{\mc{O}}_{\mc{X},x}\otimes_{\hat{\mc{O}}_{\mc{X},y}} \ms{H}(\bk{y})\cong \bigoplus_{i=1}^r (\ms{O}_{X,\bk{x}_i}/\mf{m}_\bk{y}\ms{O}_{X,\bk{x}_i})\otimes_{\kappa(\bk{y})} \ms{H}(\bk{y}),\] so $R = \sum_{i=1}^r m_f(\bk{x}_i)$. This completes the proof.
\end{proof}

\section{Equidistribution for maps with good reduction}

In this section we will prove the equidistribution theorem for maps with good reduction. The setup for this section is as follows. We fix an irreducible projective variety $X$ over $K$, and a polarized morphism $f\colon (X, L)\to (X,L)$ of algebraic degree $d\geq 2$, which has good reduction with respect to a given polarized morphism $F\colon (\mc{X}, \mc{L})\to (\mc{X}, \mc{L})$ which models $f$. We denote by $\tilde{f}$ the reduction $\tilde{f}\colon (\mc{X}_k, \mc{L}_k)\to (\mc{X}_k,\mc{L}_k)$ of $f$. When $K$ is equipped with the trivial absolute value, one has $(\mc{X}_k,\mc{L}_k) = (X,L)$ and $\tilde{f} = f$, but in the interest of keeping notation uniform we will still write $\tilde{f}$ and $(\mc{X}_k,\mc{L}_k)$.

The idea of the proof of the equidistribution theorem is to apply the results of \S5 to $\tilde{f}$, and then lift these results to $f$ via the reduction semiconjugacy: 

\bigskip

\[\begin{psmatrix}[colsep = 1 in, rowsep = .5in]
X^\an & X^\an\\ 
\mc{X}_k & \mc{X}_k
\psset{arrows=->, nodesep = 3pt}
\ncline{1,1}{1,2}^f
\ncline{2,1}{2,2}^{\tilde{f}}
\ncline{1,1}{2,1}
\tlput{\bk{red}}
\ncline{1,2}{2,2}
\trput{\bk{red}}
\end{psmatrix}\]

\bigskip\smallskip

\noindent In order to make use of the results of \S5, we will additionally need to assume that the reduced map $\tilde{f}\colon (\mc{X}_k, \mc{L}_k)\to(\mc{X}_k, \mc{L}_k)$ satisfies \hyperref[ass:technical]{Assumption~\ref*{ass:technical}}. 

We begin by using the maps $\pi\colon X^\an\to X$ and $\bk{red}\colon X^\an\to \mc{X}_k$ to relate the measure theory of $X^\an$ to the measure theory of $X$ and $\mc{X}_k$. Recall that $\pi$ and $\bk{red}$ are continuous and anticontinuous, respectively, so both of these maps are Borel measurable. If $\mu$ is a Radon measure on $X^\an$, we are therefore able to consider the push-forward measures $\pi_*\mu$ and $\bk{red}_*\mu$. Our first goal is to prove that the push-forward operations $\pi_*$ and $\bk{red}_*$ are compatible with pull-backs in the sense that $\pi_*f^* = f^*\pi_*$ and $\bk{red}_*f^* = \tilde{f}^*\bk{red}_*$.

\begin{lem}\label{lem:approx}$\,$\begin{enumerate} 
\item[$1.$] Let $V\subset \mc{X}_k$ be a nonempty proper irreducible closed subset, and let $U = \bk{red}^{-1}(V)$. There is an increasing sequence of nonnegative continuous functions $\varphi_n\colon X^\an\to \R$ which converge pointwise to the characteristic function $\chi_{U}$.
\item[$2.$] Let $V\subseteq X$ be a nonempty proper irreducible Zariski closed set, and let $E = \pi^{-1}(V)$. There is a decreasing sequence of nonnegative continuous functions $\psi_n\colon X^\an\to \R$ which converge pointwise to the characteristic function $\chi_E$.
\end{enumerate}
\end{lem}
\begin{proof} (1) Let $\{\mc{W}_\alpha\}$ be a finite affine open cover of $\mc{X}$, say $\mc{W}_\alpha = \Spec A_\alpha$. For each index $\alpha$, the set $\ms{W}_\alpha := \bk{red}^{-1}(\mc{W}_{\alpha, k})$ is a closed subset of $X^\an$ (it is, in fact, an affinoid domain). Fix an $\alpha$ such that $\mc{W}_{\alpha, k}$ intersects $V$, and let $a_1,\ldots, a_r\in A_\alpha$ be elements whose images $\ol{a}_1,\ldots, \ol{a}_r$ in the reduction $\ol{A} = A\otimes_{K^\circ} k$ generate the prime ideal $\mf{p}_V$ of $V$ in $\mc{W}_{\alpha, k}$. Then $\bk{x}\in \ms{W}_\alpha$ lies in $U$ if and only if $|a_i(\bk{x})|<1$ for each $i$. Define $h_\alpha\colon \ms{W}_\alpha\to \R$ by $h_\alpha(\bk{x}) := \max_i |a_i(\bk{x})|$. This $h_\alpha$ is continuous, and satisfies $h_\alpha(\bk{x})\leq 1$, with equality if and only if $\bk{x}\notin U$. Moreover, $h_\alpha$ is independent of the choice of the $a_i$. In this way we define $h_\alpha$ for each index $\alpha$. One has $h_\alpha  = h_\beta$ on $\ms{W}_\alpha\cap \ms{W}_\beta$, and hence the $h_\alpha$ can be glued together to give a continuous function $h\colon X^\an\to \R$ with the property that $0\leq h(\bk{x})\leq 1$ for all $\bk{x}$, with $h(\bk{x})<1$ if and only if $\bk{x}\in U$. We can then define $\varphi_n := (1 - h)^{1/n}$.

(2) Let $\{U_\alpha\}$ be a finite affine open cover of $X$, say with $U_\alpha  = \Spec A_\alpha$. For each index $\alpha$, the set $U^\an_\alpha = \pi^{-1}(U_\alpha)$ is an open subset of $X^\an$. Fix an $\alpha$ such that $U_\alpha$ intersects $V$, and let $a_1,\ldots, a_r\in A_\alpha$ be generators of the prime ideal $\mf{p}_V$ of $V$ in $U_\alpha$. Then $\bk{x}\in U^\an_\alpha$ belongs to $E$ if and only if $|a_i(\bk{x})| = 0$ for each $i$. Let $h_\alpha\colon U^\an_\alpha\to \R$ be the function $h_\alpha(\bk{x}) := \min\{1, \max_i |a_i(\bk{x})|\}$. This $h_\alpha$ is continuous, and satisfies $h_\alpha(\bk{x})\geq 0$, with equality if and only if $\bk{x}\in E$. Moreover, $h_\alpha$ is independent of the choice of the $a_i$. In this way we define $h_\alpha$ for each index $\alpha$. One has $h_\alpha = h_\beta$ on $U_\alpha^\an\cap U_\beta^\an$, and hence the $h_\alpha$ can be glued together to give a continuous function $h\colon X^\an \to \R$ with the property that $0\leq h(\bk{x})\leq 1$ for all $\bk{x}$, with $h(\bk{x}) = 0$ if and only if $\bk{x}\in E$. We can then define $\psi_n = 1 - h^{1/n}$.
\end{proof}

\begin{prop}\label{prop:compatibility} Let $\mu$ be a Radon measure on $X^\an$. Then $\pi_*$ and $\bk{red}_*$ are compatible with pull-backs in the sense that \begin{enumerate}
\item[$1.$] $\bk{red}_*f^*\mu = \tilde{f}^*\bk{red}_*\mu$, and
\item[$2.$] $\pi_*f^*\mu = f^*\pi_*\mu$.
\end{enumerate}
\end{prop}
\begin{proof} (1) It suffices to check that $(\bk{red}_*f^*\mu)(V) = (\tilde{f}^*\bk{red}_*\mu)(V)$ for every irreducible closed set $V\subset \mf{X}_0$ by \cite[Lemma 2.7]{me}. Choose an increasing sequence $\varphi_n\colon X^\an\to \R$ of non-negative continuous functions converging pointwise to $\chi_{\bk{red}^{-1}(V)}$. Then \[
(\bk{red}_*f^*\mu)(V) = (f^*\mu)(\bk{red}^{-1}(V)) = \lim_{n\to \infty} \int \varphi_n\,df^*\mu = \lim_{n\to \infty} \int f_*\varphi_n\,d\mu = \int f_*\bk{red}^*\chi_V\,d\mu.\] On the other hand, we have \[
(\tilde{f}^*\bk{red}_*\mu)(V) = \int \tilde{f}_*\chi_V\,d\bk{red}_*\mu = \int \bk{red}^*\tilde{f}_*\chi_V\,d\mu.\] Thus it suffices to show that $f_*\bk{red}^*\chi_V = \bk{red}^*\tilde{f}_*\chi_V$. If $\bk{y}\in X^\an$, then \begin{align*}
(f_*\bk{red}^*\chi_V)(\bk{y}) & = \sum_{f(\bk{x}) = \bk{y}} m_f(\bk{x})\chi_V(\bk{red}(\bk{x})) = \sum_{\tilde{f}(x) = \bk{red}(\bk{y})} m_{\tilde{f}}(x)\chi_V(x) = (\bk{red}^*\tilde{f}_*\chi_V)(\bk{y}),
\end{align*} where the second equality is a consequence of \hyperref[prop:reduced_mults]{Proposition~\ref*{prop:reduced_mults}}. This proves (1). The proof of (2) is similar, except one uses \hyperref[prop:kernel_multiplicities]{Proposition~\ref*{prop:kernel_multiplicities}} instead of \hyperref[prop:reduced_mults]{Proposition~\ref*{prop:reduced_mults}}.
\end{proof}

Unfortunately, the push-forward operations $\pi_*$ and $\bk{red}_*$ on Radon measures are \emph{not} weakly continuous. Specifically, if $\mu_n$ is a sequence of Radon measures on $X^\an$ which converge weakly to a measure $\mu$, then it is not necessarily the case that $\pi_*\mu_n$ converges weakly to $\pi_*\mu$, or that $\bk{red}_*\mu_n$ converges weakly to $\bk{red}_*\mu$. Indeed, it is not even necessarily the case that $\pi_*\mu_n$ and $\bk{red}_*\mu_n$ converge weakly to anything. The reason for this difficulty is that the weak topology for measures on $X$ and $\mc{X}_k$ is defined in terms of semicontinuous functions, whereas the weak topology for Radon measures on $X^\an$ is defined in terms of continuous functions. The next proposition explores this phenomenon.

\begin{prop}\label{prop:limit_measures} Let $\mu_n$ be a sequence of Radon probability measures on $X^\an$ which converges weakly to a measure $\mu$.\begin{enumerate}
\item[$1.$] Suppose the measures $\nu_n := \bk{red}_*\mu_n$ converge weakly to a measure $\nu$. Then one has the inequality $\nu(V)\geq (\bk{red}_*\mu)(V)$ for all irreducible closed subsets $V\subseteq \mc{X}_k$.
\item[$2.$] Suppose the measures $\nu_n := \pi_*\mu_n$ converge weakly to a measure $\nu$. Then one has the inequality $\nu(V)\leq (\pi_*\mu)(V)$ for all irreducible closed subsets $V\subseteq X$.
\end{enumerate}
\end{prop}
\begin{proof} (1) Fix an irreducible closed subset $V\subseteq \mc{X}_k$, and let $\varphi_n\colon X^\an \to \R$ be an increasing sequence of nonnegative continuous functions converging pointwise to $\chi_{\bk{red}^{-1}(V)}$. Given an $\eps>0$ and an index $N = N(\eps)$ large enough, one then has \begin{align*}
(\bk{red}_*\mu)(V) & = \lim_{n\to \infty} \int \varphi_n\,d\mu \leq \eps + \int \varphi_N\,d\mu = \eps + \lim_{m\to \infty}\int \varphi_N\,d\mu_m\\
& \leq \eps + \liminf_{m\to \infty} \int\chi_{\bk{red}^{-1}(V)}\,d\mu_m = \eps + \liminf_{m\to \infty}\nu_m(V) = \eps + \nu(V).
\end{align*} Letting $\eps\to 0$ gives $(\bk{red}_*\mu)(V) \leq \nu(V)$, as desired. The proof of (2) is similar.
\end{proof}

We are now in a position to prove our main equidistribution of preimages theorem for maps of good reduction. \hyperref[thmA]{Theorem~\ref*{thmA}} is a special case of the following.

\begin{thm}[Equidistribution]\label{equid} Let $X$ be an irreducible projective variety over $K$ of dimension $m$, and let $f\colon (X, L)\to (X,L)$ be a polarized morphism of algebraic degree $d\geq 2$. Suppose that $f$ has good reduction with respect to a morphism $F\colon (\mc{X}, \mc{L})\to (\mc{X}, \mc{L})$. Finally, assume that the reduction $\tilde{f}\colon (\mc{X}_k, \mc{L}_k)\to (\mc{X}_k, \mc{L}_k)$ of $f$ satisfies \hyperref[ass:technical]{Assumption~\ref*{ass:technical}}. Let $\ms{E}$ be the exceptional set of $\tilde{f}$. If $\mu$ is a Radon probability measure on $X^\an$ which gives no mass to $\bk{red}^{-1}(\ms{E})$, then the normalized pull-backs $d^{-mn}f^{n*}\mu$ converge weakly to the Dirac probability measure $\delta_\bk{x}$ supported at the unique point $\bk{x}\in X^\an$ whose reduction is the generic point of $\mc{X}_k$.
\end{thm}
\begin{proof} Let $\mu_n = d^{-mn}f^{n*}\mu$ for each $n\geq 1$. It suffices to show that every weakly convergent subsequence of $\{\mu_n\}$ converges to $\delta_\bk{x}$. We therefore fix a weakly convergent subsequence $\mu_{n_i}$, converging to some measure $\alpha$. Let $\nu = \bk{red}_*\mu$ and $\nu_n = \bk{red}_*\mu_n$ for each $n$. We then know from \hyperref[prop:compatibility]{Proposition~\ref*{prop:compatibility}} that $\nu_n = d^{-mn}\tilde{f}^{n*}\nu$ for each $n$. The assumption that $\mu$ does not give mass to $\bk{red}^{-1}(\ms{E})$ is equivalent $\nu$ not giving mass to $\ms{E}$. It then follows from \hyperref[cor:classical_equid]{Corollary~\ref*{cor:classical_equid}} that the sequence $\nu_n$ converges weakly to the Dirac probability measure at the generic point of $\mc{X}_k$. From \hyperref[prop:limit_measures]{Proposition~\ref*{prop:limit_measures}(1)} we see that $(\bk{red}_*\alpha)(V) = 0$ for all proper closed subsets $V\subsetneq \mc{X}_k$. We will use this property of $\alpha$ to conclude that $\alpha = \delta_\bk{x}$.

Let $\mc{A}\subseteq C^0(X^\an)$ be the subalgebra consisting of functions which are constant away from a set of the form $\bk{red}^{-1}(V)$ for some proper closed set $V\subsetneq \mc{X}_k$. The functions which were constructed in the proof of \hyperref[lem:approx]{Lemma~\ref*{lem:approx}} show that $\mc{A}$ separates points. Clearly $\mc{A}$ contains all constant functions. Thus, by the Stone-Weierstrass theorem, $\mc{A}$ is dense in $C^0(X^\an)$. Let $\varphi\in \mc{A}$, with say $\varphi\equiv c$ away from a set $\bk{red}^{-1}(V)$ with $V\subsetneq \mc{X}_k$ closed. Then \[
\int \varphi\,d\alpha = c[1 - \alpha(\bk{red}^{-1}(V))] + \int_{\bk{red}^{-1}(V)} \varphi\,d\alpha = c,\] since $\alpha(\bk{red}^{-1}(V)) = 0$. Since $\bk{x}\notin \bk{red}^{-1}(V)$, one has $\varphi(\bk{x}) = c$. We have thus shown that $\alpha$ agrees with $\delta_\bk{x}$ on $\mc{A}$. Since $\mc{A}$ is dense in $C^0(X^\an)$, we conclude that $\alpha = \delta_\bk{x}$. 
\end{proof}

In the special case where $K$ is equipped with the trivial absolute value, one can use the canonical map $\pi\colon X^\an \to X$ to obtain a more precise result about what happens to preimages of points $\bk{x}\in X^\an$ which \emph{do} lie in $\bk{red}^{-1}(\ms{E})$.

\begin{thm} Suppose $K$ is equipped with the trivial absolute value. Let $X$ be an irreducible projective variety of dimension $m$ over $K$, and $f\colon (X,L)\to (X,L)$ a flat polarized morphism of algebraic degree $d\geq 2$. Suppose that $f$ satisfies \hyperref[ass:technical]{Assumption~\ref*{ass:technical}}. Let $\bk{x}\in X^\an$ be any point, and assume that the smallest totally invariant closed set of $X$ containing $\pi(\bk{x})$ is the same as the smallest totally invariant closed set of $X$ containing $\bk{red}(\bk{x})$. Let this set $V$ have irreducible decomposition $V = V_0\cup\cdots\cup V_s$. Let $\bk{y}_i$ denote the classical point of $X^\an$ corresponding to $V_i$ for each $i$. Then, up to relabeling the $V_i$ if necessary, one has for each $i = 0,\ldots, s-1$ that $d^{-m(i+sn)}f^{(i+sn)*}\delta_\bk{x}\to \delta_{\bk{y}_i}$ weakly as $n\to \infty$.
\end{thm}
\begin{proof} We argue in a similar fashion to the proof of \hyperref[equid]{Theorem~\ref*{equid}}. Specifically, we set $\mu_n = d^{-m(i+sn)}f^{(i+sn)*}\delta_\bk{x}$, and let $\alpha$ be a weak limit of a subsequence of the $\mu_{n_i}$. Now by applying \hyperref[cor:varequid]{Corollary~\ref*{cor:varequid}} we know that the sequences $\pi_*\mu_n$ and $\bk{red}_*\mu_n$ both converge weakly to $\delta_{\bk{y}_i}$. Thus \hyperref[prop:limit_measures]{Proposition~\ref*{prop:limit_measures}} tells us that $\alpha(A) = 0$ for all sets $A$ of the form $A = \pi^{-1}(U)\cup \bk{red}^{-1}(W)$ where $U\subseteq X$ is an open set disjoint from $V_i$ and $W$ is a proper closed subset of $V_i$. Let $\mc{A}$ denote the subalgebra of $C^0(X^\an)$ consisting of all functions which are constant away from a set $A = \pi^{-1}(U)\cup \bk{red}^{-1}(W)$, where $U\subseteq X$ is an open set disjoint from $V_i$ and $W$ is a proper closed subset of $V_i$. From the functions constructed in the proof of \hyperref[lem:approx]{Lemma~\ref*{lem:approx}} we see that $\mc{A}$ separates points, and hence is dense in $C^0(X^\an)$ by the Stone-Weierstrass theorem. If $\varphi\in \mc{A}$ is such that $\varphi\equiv c$ outside of $A = \pi^{-1}(U)\cup \bk{red}^{-1}(W)$, then \[
\int \varphi\,d\alpha = c[1 - \alpha(A)] + \int_A\varphi\,d\alpha = c = \varphi(\bk{y}_i),\] and hence $\alpha$ agrees with $\delta_{\bk{y}_i}$ on the dense subalgebra $\mc{A}$. We conclude that $\alpha = \delta_{\bk{y}_i}$.
\end{proof}

\section{Equidistribution for tame points}

In this last section, we assume $K$ is an algebraically closed field equipped with the trivial absolute value. Let $X$ be an irreducible projective variety of dimension $m$ over $K$, and let $f\colon (X,L)\to (X,L)$ be a flat polarized dynamical system of algebraic degree $d\geq 2$, i.e., a map of good reduction. The goal of this section is to show that the preimages of  a large class of points $\bk{x}\in X^\an$ equidistribute to the Dirac mass at the Gauss point of $X^\an$, even when $\bk{red}(\bk{x})$ lies in the exceptional set $\ms{E}$ of $f$.

In the setting where $K$ is trivially valued, it is common to work not with seminorms but with semivaluations. A point $\bk{x}\in X^\an$ is a seminorm on the coordinate ring $K[U]$ of some affine open subset $U\subseteq X$. One may then associate to $\bk{x}$ a semivaluation $K[U]\to \R\cup\{+\infty\}$ given by $\varphi\mapsto -\log |\varphi(\bk{x})|$. Conversely, from a semivaluation $v\colon K[U]\to \R\cup\{+\infty\}$, one obtains a seminorm $\bk{x}\in X^\an$ by $|\varphi(\bk{x})| = e^{-v(\varphi)}$ for all $\varphi\in K[U]$. Via this equivalence, we will view points in $X^\an$ as being semivaluations rather than seminorms. The reduction $\bk{red}(v)$ of a semivaluation $v\in X^\an$ is then the unique point $\xi\in X$ such that $v\geq 0$ on $\mc{O}_{X,\xi}$ with $v>0$ on $\mf{m}_\xi$. 

\begin{Def} Let $v\in X^\an$ be a semivaluation with  $\bk{red}(v) = \xi\in X$. For any ideal $\mf{a}\subseteq \mc{O}_{X,\xi}$ we define $v(\mf{a}) = \inf_{\varphi\in \mf{a}} v(\varphi)$. Equivalently, if $\varphi_1,\ldots, \varphi_r$ generate $\mf{a}$, then $v(\mf{a}) = \min_i v(\varphi_i)$.
\end{Def}

We begin with a proof of \hyperref[thmC]{Theorem~\ref*{thmC}}, which proves that a weaker form of equidistribution of preimages holds for $f$ whenever all totally invariant cycles are superattracting.

\begin{Def} Suppose $V\subsetneq X$ is an $f$-invariant irreducible subvariety. We say that $V$ is \emph{superattracting} for $f$ if there is an integer $n\geq 1$ such that $f^{n*}\mf{m}_V\subseteq\mf{m}_V^2$, where here $f^{n*}$ is the induced local ring homomorphism $f^{n*}\colon \mc{O}_{X,V}\to \mc{O}_{X,V}$. More generally, if $V$ is part of an $s$-periodic cycle for $f$, we say that the cycle is superattracting for $f$ if $V$ is superattracting for $f^s$.
\end{Def}

This notion of superattracting cycle generalizes the standard notion of superattracting cycles for rational maps $f\colon \pr^1_K\to \pr^1_K$. For instance, when $V = x$ is a fixed closed point of $X$, then $x$ is superattracting if and only if the derivative $Df(x)$ of $f$ at $x$ is nilpotent. In dimension 1, it is automatic that any totally invariant point is superattracting, but this is no longer true in higher dimensions, as was first noted by Forn\ae ss-Sibony \cite[p.\ 212]{MR1285389}. As an example, the point $[0:0:1]$ is totally invariant for the morphism $f\colon \pr^2_K\to \pr^2_K$ given by $f[x:y:z] = [xz+y^2:x^2:z^2]$, but it is not superattracting. 

\begin{thm} Let $f\colon (X,L)\to (X,L)$ be a flat polarized morphism of algebraic degree $d\geq 2$.  Suppose that $f$ satisfies \hyperref[ass:technical]{Assumption~\ref*{ass:technical}}, and that all totally invariant cycles for $f$ are superattracting. Let $v\in X^\an$, and let $V\subseteq X$ be the smallest totally invariant closed set  containing $\pi(v)$. Suppose $V$ has irreducible decomposition $V = V_1\cup\cdots\cup V_r$, and let $w_1,\ldots, w_r\in X^\an$ be the classical points corresponding to the generic points of the $V_i$. Then one has weak convergence of the Cesaro means \[\mu_n := n^{-1}\sum_{i=0}^{n-1} d^{-mi}f^{i*}\delta_v\to r^{-1}(\delta_{w_1} + \cdots + \delta_{w_r})\] as $n\to \infty$. In the special case where $\pi(v)\notin \ms{E}$, this gives that the $\mu_n$ converge weakly to the Dirac mass at the Gauss point.
\end{thm}
\begin{proof} It suffices to show that any weak limit $\mu$ of a subsequence of the $\mu_n$ is the measure $r^{-1}(\delta_{w_1} + \cdots + \delta_{w_r})$. Suppose then that $\mu$ is the weak limit of a subsequence $\mu_{n_i}$. The measure $\mu$ is necessarily \emph{totally invariant}, that is,  $d^{-m}f^*\mu = \mu$. Indeed \[
d^{-m}f^*\mu = \lim_{i\to \infty} d^{-m}f^*\mu_{n_i} = \lim_{i\to \infty} [\mu_{n_i} + n_i^{-1}(d^{-mn_i}f^{n_i*}\delta_v - \delta_v)] = \lim_{i\to \infty} \mu_{n_i} = \mu.\] Since $\pi^{-1}(V)$ is a totally invariant closed subset of $X^\an$ and $\delta_v(\pi^{-1}(V)) = 1$, the measure $\mu$ is supported within $\pi^{-1}(V)$. If we can show that $\mu(\bk{red}^{-1}(W)) = 0$ for all proper totally invariant closed sets $W\subsetneq V$, then an easy variant of \hyperref[cor:classical_equid]{Corollary~\ref*{cor:classical_equid}} shows $d^{-mn}f^{n*}\mu\to r^{-1}(\delta_{w_1} + \cdots + \delta_{w_r})$. However, $d^{-mn}f^{n*}\mu = \mu$ for all $n$, so in fact $\mu = r^{-1}(\delta_{w_1} + \cdots + \delta_{w_r})$. We have therefore reduced the problem to showing that $\mu(\bk{red}^{-1}(W)) = 0$ for all irreducible closed sets $W\subsetneq V$ which are part of a totally invariant cycle for $f$.

Fix $W\subsetneq V$ an irreducible closed set with $f^{-s}(W) = W$. Let $\mf{m}_W$ be the maximal ideal of $W$ in the local ring $\mc{O}_{X,W}$. Let $\varphi\colon X^\an\to \R\cup\{+\infty\}$ be the continuous function \[
\varphi(w) := \begin{cases} w(\mf{m}_W) & \bk{red}(w)\in W.\\ 0 & \bk{red}(w)\notin W.\end{cases}\] Since $\varphi$ is strictly positive on $\bk{red}^{-1}(W)$ and zero everywhere else, one has $\mu(\bk{red}^{-1}(W)) = 0$ if and only if  $\int \varphi \,d\mu = 0$. By assumption $W$ is superattracting, i.e., $f^{ns*}\mf{m}_W\subseteq \mf{m}_W^2$ for large enough $n$. If $w\notin \bk{red}^{-1}(W)$, then all $f^{ns}$-preimages of $w$ also do not lie in $\bk{red}^{-1}(W)$, so $d^{-mns}(f^{ns}_*\varphi)(w) = 0$. On the other hand, if $\bk{red}(w)\in W$ and $f^{ns}(w')= w$, then \[
w(\mf{m}_W) = w'(f^{ns*}\mf{m}_W)\geq 2w'(\mf{m}_W)\] when $n$ is large enough. It follows that $d^{-mns}(f^{ns}_*\varphi)(w) \leq \varphi(w)/2$. Combining these two derivations yields $d^{-mns}f^{ns}_*\varphi\leq \varphi/2$ when $n$ is large. Therefore \[
0\leq \int \varphi\,d\mu = d^{-mns}\int f^{ns}_*\varphi\,d\mu \leq \frac{1}{2}\int\varphi\,d\mu.\] This is only possible if $\int\varphi\,d\mu = 0$ or $+\infty$. The latter case cannot happen, however, since by assumption $\pi(v)\notin W$.
\end{proof}

To prove \hyperref[thmB]{Theorem~\ref*{thmB}}, we need to restrict ourselves to the case when $X$ is smooth. In this case, any finite surjective morphism $f\colon X\to X$ is flat, and hence any polarized morphism $f\colon (X,L)\to (X,L)$ has good reduction. The motivation for assuming $X$ is smooth is that we are able to make the following definition.

\begin{Def} Let $\xi\in X$. Because $X$ is smooth, the local ring $\mc{O}_{X,\xi}$ is regular, and hence the map $\ord_\xi\colon \mc{O}_{X,\xi}\to \N\cup\{+\infty\}$ given by $\ord_\xi(\varphi) := \max\{n : \varphi\in \mf{m}_\xi^n\}$ is a valuation on the ring $\mc{O}_{X,\xi}$. 
\end{Def}

We will prove equidistribution of preimages for valuations $v\in X^\an$ that are \emph{tame}, that is, valuations which satisfy the following boundedness condition.

\begin{Def}\label{def:tame} Let $v\in X^\an$ with $\xi = \bk{red}(v)$. We say that $v$ is \emph{tame} if there is a $C>0$ such that $v(\varphi)\leq C\ord_\xi(\varphi)$ for all $\varphi\in \mc{O}_{X,\xi}$. Note, in particular, that such a $v$ is necessarily a \emph{valuation} instead of just a semivaluation, i.e., $\pi(v)$ is the generic point of $X$.
\end{Def}

\begin{ex} Tame valuations make up much of the space $X^\an$, as the following examples illustrate. \begin{enumerate}
\item[1.] The set of tame valuations of $\pr^{1,\an}_K$ is the \emph{hyperbolic space} $\mathbf{H}:= \pr^{1,\an}_K\smallsetminus \pr^1_K$.
\item[2.] All monomial valuations are tame. Recall that a valuation $v\in X^\an$ with $\xi = \bk{red}(v)$ is a \emph{monomial valuation} if it is of the following form: for some system of parameters $t_1,\ldots, t_r$ of the completed local ring $\hat{\mc{O}}_{X,\xi}\cong \kappa(\xi)\llbracket t_1,\ldots, t_r\rrbracket$ and some real numbers $\alpha_1,\ldots, \alpha_r>0$, one has \[
v\left({\sum}_{\beta\in \N^r} \lambda_\beta t^\beta\right) = \min\{\beta_1\alpha_1 + \cdots + \beta_r\alpha_r : \lambda_\beta\neq 0\}.\] It is easy to check that such a valuation is tame.
\item[3.] Divisorial valuations are tame. A valuation $v\in X^\an$ is \emph{divisorial} if there is a blowup $p\colon X'\to X$, an exceptional prime divisor $E$ of $p$, and a real number $\lambda>0$ such that $v(\varphi) = \lambda\ord_E(\varphi\circ p)$ for all $\varphi\in K(X)$.
\item[4.] More generally, all quasimonomial valuations are tame. A valuation $v\in X^\an$  is \emph{quasimonomial} if there is some blowup $p\colon X'\to X$ over $\xi$ and a monomial valuation $w$ at a point $\zeta\in X'$ such that $v(\varphi) = w(\varphi\circ p)$ for all $\varphi\in K(X)$. Such valuations are studied in detail in \cite{jonsson-mustata:asymptotic}, see also \cite{MR1963690}. Monomial valuations and divisorial valuations are both examples of quasimonomial valuations. The tameness of quasimonomial valuations was proved by Tougeron (\cite[Lemma IX.1.3]{MR0440598}, see also \cite{MR817161}). They are sometimes called \emph{Abhyankar valuations}, since they are precisely those valuations $v\in X^\an$ for which one has equality in the Abhyankar inequality; they are further characterized by being Shilov boundaries of Weierstrass domains in $X^\an$. See \cite[\S4]{Poineau} for more details about these alternate characterizations. Quasimonomial valuations (indeed, even divisorial valuations) are dense in $X^\an$.
\item[5.] One should note that there are tame valuations in $X^\an$ that are not quasimonomial, and, if $\dim X>1$, there are valuations in $X^\an$ that are not tame, see for instance \cite[Proposition A.3]{MR2097722}.
\end{enumerate}
\end{ex}

Suppose that $v\in X^\an$ is a valuation, and thus that it defines a valuation on the function field $K(X)$ of $X$. If we identify $K(X)$ with the subfield $f^*K(X)\subseteq K(X)$, then the preimages of $v$ for $f$ are precisely the valuations on $K(X)$ which extend $v$ on $f^*K(X)$. For the remainder of the section, we will assume that the map $f$ is \emph{separable}, that is, that the field extension $K(X)/f^*K(X)$ is separable. This is a weaker assumption than \hyperref[ass:technical]{Assumption~\ref*{ass:technical}}.%If $f(w) = v$, we denote by $\bk{e}(w/v)$ and $\bk{f}(w/v)$ the ramification index and relative degree of the extension $w$ of $v$, respectively.

\begin{prop} Suppose that $v\in X^\an$ is a valuation, and $f(w) = v$. Let $L/F$ denote the field extension $K(X)/f^*K(X)$, so that $w$ is a valuation on $L$ extending $v$ on $F$. If $f$ is separable, then the multiplicity of $f$ at $w$ is given by $m_f(w) = [L_w : F_v]$, where $L_w$ and $F_v$ denote the completions of $L$ and $F$ with respect to $w$ and $v$, respectively.
\end{prop}
\begin{proof} The local rings $\ms{O}_{X,w}$ and $\ms{O}_{X,v}$ are fields, and $m_f(w) = [\ms{O}_{X,w}:\ms{O}_{X,v}]$ for any preimage $w$ of $v$.  If $\ms{H}(w)$ and $\ms{H}(v)$ denote the completed residue fields of $w$ and $v$, respectively, then $[\ms{H}(w): \ms{H}(v)]\leq [\ms{O}_{X,w} : \ms{O}_{X,v}]$. One has isomorphisms $\ms{H}(w) \cong L_w$ and $\ms{H}(v)\cong F_v$, and therefore $[L_w:F_v] \leq [\ms{O}_{X,w} : \ms{O}_{X,v}] = m_f(w)$. Since $L/F$ is separable, \[
F_v\otimes_F L\cong \bigoplus_{f(w) = v} L_w\] by \cite[Corollary VI.8.2/2]{MR1727221}, and thus \[
d^m = [L:F] = \sum_{f(w) = v} [L_w:F_v]\leq \sum_{f(w) = v} m_f(w) = d^m.\] It follows that $m_f(w) = [L_w:F_v]$ for each preimage $w$ of $v$.
\end{proof}

\begin{cor}\label{cor:bourbaki} Suppose $f$ is a separable, and let $N_f$ denote the norm homomorphism $N_f\colon K(X)^\times\to f^*K(X)^\times$ associated to the field extension $K(X)/f^*K(X)$. If $v\in X^\an$ is a valuation and $\varphi\in K(X)^\times$, then \[
\sum_{f(w) = v} m_f(w)w(\varphi) = v(N_f(\varphi)).\]
\end{cor}
\begin{proof} See \cite[Corollary VI.8.5/3]{MR1727221}.
\end{proof}

\begin{prop}\label{prop:eq_char} Let $v\in X^\an$ be a valuation. Assume that for each $\varphi\in K(X)^\times$ one has \[
d^{-mn}\sum_{f^n(w)  =v} m_{f^n}(w)|w(\varphi)|\to 0\] as $n\to \infty$. Then the preimages of $v$ equidistribute to the Dirac mass at the Gauss point.
\end{prop}
\begin{proof} Let $\mu_n = d^{-mn}f^{n*}\delta_v$ for each $n\geq 0$. It suffices to show that every weak limit $\mu$ of a subsequence $\mu_{n_i}$ is the Dirac mass at the Gauss point. Suppose such a $\mu$ were not the Dirac mass at the Gauss point. Then there is some irreducible proper closed set $E\subsetneq X$ such that $\mu(\bk{red}^{-1}(E))>0$. Let $\psi\colon X^\an\to [0,1]$ be the continuous function \[
\psi(w) := \begin{cases} \min\{1, w(\mf{m}_E)\} & \bk{red}(w)\in E.\\ 0 & \bk{red}(w)\notin E.\end{cases}\] This function is strictly positive on $\bk{red}^{-1}(E)$ and $0$ elsewhere, so $\int \psi\,d\mu >0$. On the other hand, if $\varphi\in \mf{m}_E$, then $\psi(w)\leq |w(\varphi)|$. Thus \[
0<\int \psi\,d\mu = \lim_{i\to \infty} \int \psi\,d\mu_{n_i} \leq \lim_{i\to \infty} d^{-mn_i} \sum_{f^{n_i}(w) = v} m_{f^n}(w)|w(\varphi)| = 0,\] a contradiction. This completes the proof.
\end{proof}

\begin{lem} Let $p\in X$ be a closed point, and let $D$ be an effective divisor on $X$ with local defining equation $\varphi$ at $p$. Then one has the inequality $\ord_p(\varphi)\leq \deg_{L^s}D$, where $s\geq 1$ is an integer large enough that $L^s$ is very ample.
\end{lem}
\begin{proof} The lemma is trivial if $p\notin\mathrm{Supp}(D)$, so assume $p\in \mathrm{Supp}(D)$. Since $L^s$ is very ample, there exist global sections $s_1,\ldots, s_m$ of $L^s$ vanishing at $p$ such that the $t_i := s_{i,p}\in \mf{m}_p\subset\mc{O}_{X,p}$ generate the tangent space at $p$. Replacing the $s_i$ by some $K$-linear combination of the $s_i$ if necessary, the Weierstrass preparation theorem gives that $\varphi$ can be decomposed in $\hat{\mc{O}}_{X,p}$ as $\varphi = uQ$, where $u$ is a unit and \[
Q(t) = t_m^n + g_1(t_1,\ldots, t_{m-1})t_m^{n-1} + \cdots + g_n(t_1,\ldots, t_{m-1})\] is a Weierstrass polynomial of degree $n = \ord_p(\varphi)$. It follows that \[
\dim_K \mc{O}_{X,p}/(\varphi, t_1,\ldots, t_{m-1}) = \dim_K K[t_m]/(t_m^n) = n = \ord_p(\varphi).\] On the other hand, $\dim_K \mc{O}_{X,p}/(\varphi, t_1,\ldots, t_{m-1})$ is exactly the local intersection multiplicity $D\cdot \mathrm{Div}(s_1)\cdot\cdots\cdot \mathrm{Div}(s_{m-1})$ at $p$. This is, of course, bounded above by the global intersection number $D\cdot \mathrm{Div}(s_1)\cdot\cdots\cdot \mathrm{Div}(s_{m-1}) = \deg_{L^s}D$.
\end{proof}

\begin{thm} Let $X$ be a smooth irreducible projective variety, and $f\colon (X,L)\to (X,L)$ a separable polarized morphism of degree $d\geq 2$. Let $v\in X^\an$ be a tame valuation. Then the preimages of $v$ equidistribute to the Dirac mass at the Gauss point of $X^\an$.
\end{thm}
\begin{proof}  Let $\varphi\in K(X)^\times$ be a nonconstant function, and let $D_1$ and $D_2$ be  effective divisors  on $X$ with $\mathrm{Div}(\varphi) = D_1 - D_2$. Let $\xi = \bk{red}(v)$. Fix an $n\geq 1$, and let $\psi_1,\psi_2\in K(X)^\times$ be rational functions that are regular at every $f^n$-preimage of $\xi$ such that $\varphi = \psi_1/\psi_2$. Then if  $w\in f^{-n}(v)$ is any $f^n$-preimage of $v$, \[
|w(\varphi)| = |w(\psi_1) - w(\psi_2)|\leq w(\psi_1) + w(\psi_2).\] It follows from \hyperref[cor:bourbaki]{Corollary~\ref*{cor:bourbaki}} that \[
\sum_{f^n(w) = v} m_{f^n}(w)|w(\varphi)| \leq v(N_{f^n}(\psi_1)) + v(N_{f^n}(\psi_2)).\] By construction, $\psi_i$ is gives a local defining equation of $D_i$ at each $\zeta\in f^{-n}(\xi)$ for $i = 1,2$. Since $\mathrm{Div}(N_{f^n}(\psi_i)) = f^n_*\mathrm{Div}(\psi_i)$, it follows that $N_{f^n}(\psi_i)$ is regular at $\xi$ and that it is a local defining equation for $f^n_*D_i$ at $\xi$ for $i = 1,2$. 

By assumption $v$ is tame, so there is a constant $C>0$ such that $v\leq C\ord_\xi$ on $\mc{O}_{X,\xi}$. We then get the inequality \[
\sum_{f^n(w) = v} m_{f^n}(w)|w(\varphi)| \leq C\ord_\xi(N_{f^n}(\psi_1)) + C\ord_\xi(N_{f^n}(\psi_2)).\] If $p\in X$ is a closed point specializing $\xi$ at which both $N_{f^n}(\psi_1)$ and $N_{f^n}(\psi_2)$ are regular, then $\ord_\xi(N_{f^n}(\psi_i))\leq \ord_p(N_{f^n}(\psi_i)$ for $i = 1,2$, giving \[
\sum_{f^n(w) = v} m_{f^n}(w)|w(\varphi)| \leq C\ord_p(N_{f^n}(\psi_1)) + C\ord_p(N_{f^n}(\psi_2)).\] Since $N_{f^n}(\psi_i)$ is the local defining equation of $f^n_*D_i$ at $p$, we have the estimate $\ord_p(N_{f^n}(\psi_i)) \leq \deg_{L^s}f^n_*D_i$, where $s\geq 1$ is an integer large enough that $L^s$ is very ample. Thus \begin{align*}
\sum_{f^n(w) = v}m_{f^n}(w) |w(\varphi)| &\leq C\deg_{L^s}f^n_*D_1 + C\deg_{L^s}f^n_*D_2 = C\deg_{f^{n*}L^s}D_1 + C\deg_{f^{n*}L^s}D_2\\
& = Cd^{n(m-1)}(\deg_{L^s}D_1 + \deg_{L^s}D_2).
\end{align*} This proves the estimate $d^{-mn}\sum_{f^n(w) = v} m_{f^n}(w)|w(\varphi)| = O(d^{-n})$. Using \hyperref[prop:eq_char]{Proposition~\ref*{prop:eq_char}}, the preimages of $v$ equidistribute to the Dirac mass at the Gauss point of $X^\an$.
\end{proof}

\bibliographystyle{alpha}
%\addcontentsline{toc}{chapter}{References}
\bibliography{References}

\end{document}